\documentclass{article}
\usepackage[english]{babel}
\usepackage[utf8]{inputenc}
\usepackage[T1]{fontenc}
\usepackage{amsmath,amssymb,amsfonts,amsopn,amsthm}
\usepackage[a4paper, total={6in, 8in}]{geometry}
\usepackage{mathtools}
\usepackage{algorithm}
\usepackage{algpseudocode}

\usepackage[dvipsnames]{xcolor}
\usepackage{graphicx,epstopdf} 
\usepackage[labelformat=simple]{subcaption}
\usepackage[format=plain,
textfont=it,
size=footnotesize,labelsep=period]{caption}            
\usepackage{tikz}
\usetikzlibrary{spy}
\usepackage{pgfplots}
\usepackage{siunitx,array,booktabs}
\usepackage{ifpdf} 
\usetikzlibrary{patterns}
\pgfdeclarelayer{foreground layer} 
\pgfsetlayers{main,foreground layer}

\usepackage{commath}
\usepackage{bm}
\usepackage{enumitem}
\usepackage[colorinlistoftodos,italian]{todonotes}
\usepackage[colorlinks = true,linkcolor = black,
citecolor = ForestGreen, anchorcolor = black]{hyperref}
\usepackage[capitalise,noabbrev]{cleveref}
\crefname{equation}{}{}
\Crefname{equation}{Equation}{Equations}
\usepackage{autonum}
\usepackage{calligra}
\usepackage{csvsimple}
\usepackage{import}
\usepackage{authblk}

\pgfplotsset{compat=1.11}

\newcommand{\plotlinewidth}{1.0}

\newcommand{\plotmarksizeu}{2.5}

\newcommand{\plotimscale}{0.98}

\newcommand{\plotimsized}{0.9}

\newcommand{\aspectratio}{0.66}
\newcommand{\subfigsize}{0.49}

\newcommand{\subfigsizesm}{0.32}
\newcommand{\legendfontscale}{1}
\newcommand{\legendmarkscale}{1}

\pgfplotsset{select coords between index/.style 2 args={
		x filter/.code={
			\ifnum\coordindex<#1\fi
			\ifnum\coordindex>#2\fi
		}
}}

\numberwithin{equation}{section}
\newtheorem{theorem}{Theorem}[section]

\newtheorem{lemma}[theorem]{Lemma}
\theoremstyle{definition}

\newtheorem{assumption}[theorem]{Assumption}
\theoremstyle{remark}

\crefname{assumption}{Assumption}{Assumptions}
\crefname{remark}{Remark}{Remarks}
\crefname{example}{Example}{Examples}

\DeclareMathAlphabet{\mathcalligra}{T1}{calligra}{m}{n}
\DeclareFontShape{T1}{calligra}{m}{n}{<->s*[2.2]callig15}{}
\newcommand{\ar}{\mathcalligra{r}}

\newcommand{\Nb}{\mathbb{N}}
\newcommand{\Rb}{\mathbb{R}}
\newcommand{\Pb}{\mathbb{P}}

\newcommand{\Rd}{\mathbb{R}^{d}}

\newcommand{\bx}{\bm{x}}
\newcommand{\by}{\bm{y}}
\newcommand{\bv}{\bm{v}}
\newcommand{\bbeta}{\bm{\beta}}

\newcommand{\Ok}{\Omega_k}

\DeclareMathOperator{\divop}{div}
\newcommand{\OpB}{\mathcal{B}}
\newcommand{\OpI}{\mathcal{I}}

\newcommand{\nld}[1]{\Vert #1\Vert}

\newcommand{\nLd}[1]{\Vert #1 \Vert_{L^2(\Omega)}}
\newcommand{\nLdd}[1]{\Vert #1\Vert_{L^2(\Omega)^d}}
\newcommand{\nLds}[1]{\Vert #1\Vert_{L^2(\sigma)}}
\newcommand{\nLus}[1]{\Vert #1\Vert_{L^1(\sigma)}}
\newcommand{\nLdK}[1]{\Vert #1\Vert_{L^2(K)}}
\newcommand{\nLddK}[1]{\Vert #1\Vert_{L^2(K)^d}}

\newcommand{\nB}[1]{\vert\hspace{-0.25ex}\vert\hspace{-0.25ex}\vert #1\vert\hspace{-0.25ex}\vert\hspace{-0.25ex}\vert}
\newcommand{\nBp}[1]{{\nB{#1}}_{\oplus}}

\NewDocumentCommand\nLdk{mg}{\Vert #1\Vert_{L^2(\Omega_\IfNoValueTF{#2}{k}{#2})}}
\NewDocumentCommand\nLddk{mg}{\Vert #1\Vert_{L^2(\Omega_\IfNoValueTF{#2}{k}{#2})^d}}

\newcommand{\F}{\mathcal{F}}
\newcommand{\Fb}{\mathcal{F}_{b}}
\newcommand{\Fi}{\mathcal{F}_{i}}
\newcommand{\FK}{\mathcal{F}_K}
\newcommand{\Fk}{\mathcal{F}_k}
\newcommand{\Fkb}{\mathcal{F}_{k,b}}
\newcommand{\Fki}{\mathcal{F}_{k,i}}
\newcommand{\hFk}{\widehat{\mathcal{F}}_k}
\newcommand{\hFki}{\widehat{\mathcal{F}}_{k,i}}

\newcommand{\M}{\mathcal{M}}
\newcommand{\Mk}{\mathcal{M}_k}
\newcommand{\VT}{V(\mathfrak{T})}

\newcommand{\jump}[1]{[\![#1]\!]}
\newcommand{\mean}[1]{\{\!\!\{#1\}\!\!\}}

\newcommand{\ns}{\bm{n}_\sigma}
\newcommand{\hMk}{\widehat{\mathcal{M}}_k}

\newcommand{\bt}{\bm{t}}
\newcommand{\bq}{\bm{q}}
\newcommand{\br}{\bm{r}}
\newcommand{\btk}{\bt_k}
\newcommand{\bqk}{\bq_k}
\newcommand{\hbtk}{\hat{\bt}_k}
\newcommand{\hbqk}{\hat{\bq}_k}
\newcommand{\hbtj}{\hat{\bt}_j}
\newcommand{\hbqj}{\hat{\bq}_j}
\newcommand{\hbrk}{\hat{\br}_k}
\newcommand{\huk}{\hat{u}_k}

\newcommand{\email}[1]{\href{#1}{#1}}

\title{A local adaptive discontinuous Galerkin method for convection-diffusion-reaction equations}

\author[1]{Assyr Abdulle}
\author[1]{Giacomo Rosilho de Souza\thanks{Corresponding author. E-mail address: \email{giacomo.rosilhodesouza@epfl.ch}.}}
\affil[1]{ANMC, Institute of Mathematics, École Polytechnique Fédérale de Lausanne, 1015 Lausanne, Switzerland}

\begin{document}

\maketitle


\begin{abstract}
We introduce a local adaptive discontinuous Galerkin method for convection-diffusion-reaction equations. The proposed method is based on a coarse grid and iteratively improves the solution's accuracy by solving local elliptic problems in refined subdomains. For purely diffusion problems, we already proved that this scheme converges under minimal regularity assumptions [A.~Abdulle and G.~{Rosilho de Souza}, {\em ESAIM: M2AN}, 53(4):1269--1303, 2019]. In this paper, we provide an algorithm for the automatic identification of the local elliptic problems' subdomains employing a flux reconstruction strategy. Reliable error estimators are derived for the local adaptive method. Numerical comparisons with a classical nonlocal adaptive algorithm illustrate the efficiency of the method.
\end{abstract}

\section{Introduction} \label{sec:intro}
Solutions to partial differential equations that exhibit singularity (e.g. cracks) or high variations in the computational domain are usually approximated by adaptive numerical methods. There is nowadays a large body of literature concerned with the development of reliable a posteriori error estimators aiming for mesh refinement in regions of large errors (see e.g. \cite{BaR78a, BaR78b, BaW85, Ver96a}). However, classical adaptive methods are usually based on iterative processes which rely on recomputing the solution on the whole computational domain for each new mesh obtained after a refinement procedure. 

In this paper we present a scheme which solves local problems defined on refined regions only. Local schemes have been proposed in the past, we mention the Local Defect Correction (LDC) method \cite{Hac84}, the Fast Adaptive Composite (FAC) grid algorithm \cite{McT86} and the Multi-Level Adaptive (MLA) technique \cite{Bra77}. At each iteration, these algorithms solve a problem on a coarse mesh on the whole domain and a local problem on a finer mesh. The coarse solution is used for artificial boundary conditions while the local solution is used to correct the residual in the coarse grid. In \cite{BRL15} the LDC scheme has been coupled with error estimators, which are used to select the local domain. 

In \cite{AbR19} we proposed a Local Discontinuous Galerkin Gradient Discretization (LDGGD) method which decomposes the computational domain in local subdomains encompassing the large gradient regions. This scheme iteratively improves a coarse solution on the full domain by solving local elliptic problems on finer meshes. 
Hence, the full problem is solved only in the first iteration on a coarse mesh while a sequence of solutions on smaller subdomains are subsequently computed. In turn iterations between subdomains are not needed as in the LDC, FAC or MLA schemes and the condition number of the small systems are considerably smaller than the one of large systems (which describe data and mesh variations on the whole domain). The LDGGD method has been shown to converge under minimal regularity assumptions, i.e. when the solution is in $H^1_0(\Omega)$ and the forcing term in $H^{-1}(\Omega)$ \cite{AbR19}. 
However, the marking of the subdomains this scheme did so far rely on the a priori knowledge of the location of high gradient regions. 

The main contribution of this paper is to propose an adaptive local LDGGD method. This adaptive method is based on a posteriori error estimators that automatically identify the subdomains to be refined. 
This is crucial for practical applications of the method. The LDGGD relies on the symmetric weighted interior penalty Galerkin (SWIPG) method \cite{PiE12,ESZ09} and we consider linear advection-diffusion-reaction equations
\begin{equation}\label{eq:elliptic}
	\begin{aligned}
	-\nabla\cdot(A\nabla u)+\bbeta\cdot\nabla u+\mu u &=f\qquad   && \text{in } \Omega, \\
	u&=0 &&  \text{in } \partial \Omega, 
	\end{aligned}
\end{equation}
where $\Omega$ is an open bounded polytopal connected subset of $\Rd$ for $d\geq 2$, $A$ is the diffusion tensor, $\bbeta$ the velocity field, $\mu$ the reaction coefficient and $f$ a forcing term. 
In \cite{ESV10} the authors introduce a posteriori error estimators for the SWIPG scheme based on cutoff functions and conforming flux and potential reconstructions, these estimators are shown to be efficient and robust in singularly perturbed regimes. Following the same strategy, we derive estimators for the local scheme by weakening the regularity requirements on the reconstructed fluxes. The new estimators are as well free of unknown constants and their robustness is verified numerically. Furthermore, they are employed to define the local subdomains and provide error bounds on the numerical solution of the LDGGD method.
We prove that the error estimators are reliable. Because of the local nature of our scheme, we introduce two new estimators that measure the jumps at the boundaries of the local domains. However, these two new terms have lower convergence rate than the other terms and we cannot establish the efficiency of our a posteriori estimators with our current approach.
Nevertheless, the two new terms are useful in our algorithm: whenever the errors are localized these new terms become negligible; in contrast, when these estimators dominate it is an indication that the error is not localized and one can switch to a nonlocal method. Other boundary conditions than those of \cref{eq:elliptic} can be considered, at the cost of modifying the error estimators introduced in \cite{ESV10}. The new estimators introduced here need no changes.

The outline of the paper is as follows. In \cref{sec:ladg} we describe the local scheme, in \cref{sec:err_flux} we introduce the error estimators and state the main a posteriori error analysis results. \Cref{sec:errbound} is dedicated to the definition of the reconstructed fluxes and proofs of the main results. Finally, various numerical examples illustrating the efficiency, versatility and limits of the proposed method are presented in \cref{sec:num}.

\section{Local adaptive discontinuous Galerkin method}\label{sec:ladg}
In this section we introduce the local algorithm based on the discontinuous Galerkin method. 
We start by some assumptions on the data and the domain, before introducing the weak form corresponding to \eqref{eq:elliptic}.
We assume that $\Omega\subset\Rd$ is a polytopal domain with $d\geq 2$, $\bbeta\in W^{1,\infty}(\Omega)^d$, $\mu\in L^\infty(\Omega)$ and $A\in L^\infty(\Omega)^{d\times d}$, with $A(\bx)$ a symmetric piecewise constant matrix with eigenvalues in $[\underline\lambda,\overline \lambda]$, where $\overline\lambda\geq\underline\lambda>0$. Moreover, we assume that $\mu-\frac{1}{2}\nabla\cdot\bbeta\geq 0$ a.e. in $\Omega$. This term $\mu-\frac{1}{2}\nabla\cdot\bbeta$ appears in the symmetric part of the operator $\OpB(\cdot,\cdot)$ defined in \eqref{eq:bform} and hence the assumption $\mu-\frac{1}{2}\nabla\cdot\bbeta\geq 0$ is needed for coercivity. Finally, we set $f\in L^2(\Omega)$.
Under these assumptions, the unique weak solution $u\in H^1_0(\Omega)$ of \eqref{eq:elliptic} satisfies
\begin{equation}\label{eq:weak}
\OpB(u,v)=\int_\Omega fv\dif \bx\qquad \text{for all }v\in H^1_0(\Omega),
\end{equation} 
where
\begin{equation}\label{eq:bform}
\OpB(u,v)= \int_\Omega (A\nabla u\cdot\nabla v+(\bbeta\cdot\nabla u) v+\mu u v)\dif\bx.
\end{equation}

\subsection{Preliminary definitions}\label{sec:prel}
We start by collecting some notations related to the geometry and the mesh of the subdomains, before recalling the definition of the discontinuous Galerkin finite element method.

\subsubsection*{Subdomains and meshes}
Let $M\in\Nb$ and $\{\Omega_k\}_{k=1}^M$ be a sequence of open subdomains of $\Omega$ with $\Omega_1=\Omega$. The domains $\Omega_k$ for $k\geq 2$ can be any polytopal subset of $\Omega$, in practice they will be chosen by the error estimators (see \cref{sec:localg}). We consider $\{\Mk\}_{k=1}^M$ a sequence of simplicial meshes on $\Omega$ and $\Fk=\Fkb\cup\Fki$ is the set of boundary and internal faces of $\Mk$. The assumption below ensures that $\M_{k+1}$ is a refinement of $\Mk$ inside the subdomain $\Omega_{k+1}$.
\begin{assumption}\label{ass:mesh}
	$\phantom{=}$
	\begin{enumerate}
		\item For each $k=1,\ldots,M$, $\overline{\Omega}_{k}=\cup_{K\in\Mk,\,K\subset\Ok} \overline{K}$.
		\item For $k=1,\ldots,M-1$,
		\begin{enumerate}[label=\alph*)]
			\item $\{K\in\mathcal{M}_{k+1}\,:\, K\subset \Omega\setminus\Omega_{k+1}\}=\{K\in\Mk\,:\, K\subset \Omega\setminus\Omega_{k+1}\}$,
			\item if $K,T\in \Mk$ with $K\subset \Omega_{k+1}$, $T\subset\Omega\setminus\Omega_{k+1}$ and $\partial K\cap\partial T\neq\emptyset$ then $K\in \M_{k+1}$,
			\item if $K\in\Mk$ and $K\subset \Omega_{k+1}$, either $K\in \M_{k+1}$ or $K$ is a union of elements in $\M_{k+1}$.
		\end{enumerate}
	\end{enumerate}
\end{assumption}

Let $\widehat{\M}_k=\{K\in\Mk\,:\, K\subset \Ok\}$ and $\widehat{\mathcal{F}}_k=\widehat{\mathcal{F}}_{k,b}\cup \widehat{\mathcal{F}}_{k,i}$ the set of faces of $\widehat{\M}_k$, with $\widehat{\mathcal{F}}_{k,b}$ and $\widehat{\mathcal{F}}_{k,i}$ the boundary and internal faces, respectively. Condition 1 in \cref{ass:mesh} ensures that $\hMk$ is a simplicial mesh on $\Ok$. Condition 2 guarantees that in $\Omega\setminus\Omega_{k+1}$ and in the neighborhood of $\partial\Omega_{k+1}\setminus\partial\Omega$ the meshes $\Mk$ and $\M_{k+1}$ are equal and that $\M_{k+1}$ is a refinement of $\Mk$ in $\Omega_{k+1}$. An example of domains and meshes satisfying \cref{ass:mesh} is illustrated in \cref{fig:illustrationmesh}.

\begin{figure}
	\begin{center}
		\begin{tikzpicture}[xscale=1]
		\draw[step=1.0,black, thin] (0,0) grid (12,4);
		\foreach \x in {0,1,2,3,4,5,6,7,8}
			\draw[thin] (\x,0)--(\x+4,4);
		\draw[thin] (0,3)--(1,4);
		\draw[thin] (0,2)--(2,4);
		\draw[thin] (0,1)--(3,4);
		\draw[thin] (9,0)--(12,3);
		\draw[thin] (10,0)--(12,2);
		\draw[thin] (11,0)--(12,1);
		\draw[ultra thick] (0,0)--(12,0)--(12,4)--(0,4)--(0,0);
		\draw[ultra thick] (2,1)--(10,1)--(10,4)--(2,4)--(2,1);
		\draw[ultra thick] (4,2)--(8,2)--(8,4)--(4,4)--(4,2);
		
		\draw[step=0.5,black, thin] (3,2) grid (9,4);
		\draw[thin] (3,3.5)--(3.5,4);
		\draw[thin] (3,2.5)--(4.5,4);
		\draw[thin] (4.5,2.5)--(6,4);
		\draw[thin] (4.5,2.0)--(6.5,4);
		\draw[thin] (5.5,2.5)--(7,4);
		\draw[thin] (5.5,2.0)--(7.5,4);
		\draw[thin] (6.5,2.5)--(7.5,3.5);
		\draw[thin] (3.5,2.0)--(5.5,4);
		\draw[thin] (4.5,3.5)--(5.0,4);
		\draw[thin] (6.5,2.0)--(8.5,4.0);
		\draw[thin] (7.5,2.0)--(9,3.5);
		\draw[thin] (8.5,2.0)--(9,2.5);
		
		\foreach \x in {4.5,5,...,7}
		\foreach \y in {3,3.5,4}
		\draw[thin] (\x,\y)--(\x+0.5,\y-0.5);

		\node[align=right,fill=white] at (12.5,1.5) {\Large{$\Omega_1$}};
		\node[align=right,fill=white] at (10.5,2.5) {\Large{$\Omega_2$}};
		\node[align=right,fill=white] at (7.5,1.5) {\Large{$\Omega_3$}};
		\end{tikzpicture}
	\end{center}
	\caption{Example of possible meshes for three embedded domains $\Omega_1$, $\Omega_2$, $\Omega_3$.}
	\label{fig:illustrationmesh}
\end{figure}
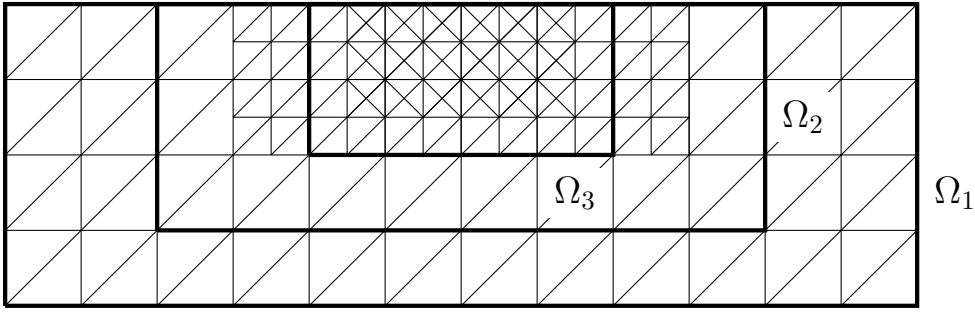

\subsubsection*{Discontinuous Galerkin finite element method}
The local adaptive discontinuous Galerkin method will solve local elliptic problems in $\Ok$ by using a discontinuous Galerkin scheme introduced in \cite{ESZ09}, which we recall here. 
In what follows, $\mathfrak{T}=(D,\M,\mathcal{F})$ denotes a tuple defined by a domain $D$, a simplicial mesh $\M$ on $D$ and its set of faces $\mathcal{F}=\mathcal{F}_b\cup\mathcal{F}_i$. In practice we will consider $\mathfrak{T}_k=(\Omega,\Mk,\mathcal{F}_k)$ or $\widehat{\mathfrak{T}}_k=(\Ok,\hMk,\hFk)$. For $\mathfrak{T}=(D,\M,\mathcal{F})$ we define 
\begin{equation}\label{eq:defVT}
V(\mathfrak{T}) = \{v\in L^2(D)\,:\, v|_K\in \Pb_\ell(K),\,\forall K\in\M\},
\end{equation}
where $\Pb_\ell(K)$ is the set of polynomials in $K$ of total degree $\ell$. As usual for such discontinuous Galerkin methods we need to define appropriate averages, jumps, weights and penalization parameters. For $K\in\M$ we denote $\bm{n}_K$ the unit normal outward to $K$ and $\FK=\{\sigma\in\F\,:\,\sigma\subset\partial K\}$. Let $\sigma\in\Fi$ and $K,T\in\M$ with $\sigma=\partial K\cap\partial T$, then $\ns=\bm{n}_K$ and
\begin{equation}
\delta_{K,\sigma}=\ns^\top A|_K\ns, \qquad\qquad \delta_{T,\sigma}=\ns^\top A|_T\ns.
\end{equation}
The weights are defined by
\begin{equation}
\omega_{K,\sigma}=\frac{\delta_{T,\sigma}}{\delta_{K,\sigma}+\delta_{T,\sigma}}, \qquad\qquad \omega_{T,\sigma}=\frac{\delta_{K,\sigma}}{\delta_{K,\sigma}+\delta_{T,\sigma}}
\end{equation}
and the penalization parameters by
\begin{equation}
\gamma_\sigma=2\frac{\delta_{K,\sigma}\delta_{T,\sigma}}{\delta_{K,\sigma}+\delta_{T,\sigma}}, \qquad\qquad \nu_\sigma=\frac{1}{2}|\bbeta\cdot \bm{n}_\sigma|.
\end{equation}
If $\sigma\in\Fb$ and $K\in\M$ with $\sigma=\partial K\cap \partial D$ then $\ns$ is $\bm{n}_D$ the unit outward normal to $\partial D$ and
\begin{equation}
\delta_{K,\sigma}=\ns^\top A|_K\ns, \qquad \omega_{K,\sigma}=1, \qquad \gamma_\sigma=\delta_{K,\sigma}, \qquad \nu_\sigma=\frac{1}{2}|\bbeta\cdot \bm{n}_\sigma|.
\end{equation}
Let $g\in L^2(\partial D)$, we define the averages and jumps of $v\in\VT$ as follows.
For $\sigma\in\Fb$ with $\sigma=\partial K\cap\partial D$ we set
\begin{equation}
\mean{v}_{\omega,\sigma}=v|_K, \qquad\qquad \mean{v}_{g,\sigma}=\frac{1}{2}(v|_K+g), \qquad\qquad \jump{v}_{g,\sigma}=v|_K-g
\end{equation}
and for $\sigma\in\Fi$ with $\sigma=\partial K\cap\partial T$
\begin{equation}
\mean{v}_{\omega,\sigma}=\omega_{K,\sigma}v|_K+\omega_{T,\sigma}v|_T, \qquad\qquad
\mean{v}_{g,\sigma}=\frac{1}{2}(v|_K+v|_T ), \qquad\qquad
\jump{v}_{g,\sigma} = v|_K-v|_T.
\end{equation}
We define $\jump{\cdot}_{\sigma}= \jump{\cdot}_{0,\sigma}$ and $\mean{\cdot}_{\sigma}= \mean{\cdot}_{0,\sigma}$. A similar notation holds for vector valued functions and whenever no confusion can arise the subscript $\sigma$ is omitted. Let $h_\sigma$ be the diameter of $\sigma$ and $\eta_\sigma>0$ a user parameter, for $u,v\in\VT$ we define the bilinear form
\begin{align}\label{eq:opB}
\begin{split}
\OpB(u,v,\mathfrak{T},g)&=
\int_{D} (A\nabla u\cdot \nabla v +(\mu-\nabla\cdot \bbeta)u v-u\bbeta\cdot \nabla v)\dif\bx\\
&\quad-\sum_{\sigma\in\F}\int_\sigma(\jump{v}\mean{A\nabla u}_{\omega}\cdot \ns+\jump{u}_{g}\mean{A\nabla v}_{\omega}\cdot \ns)\dif\by\\
&\quad+\sum_{\sigma\in\F}\int_\sigma ((\eta_\sigma\frac{\gamma_\sigma}{h_\sigma}+\nu_\sigma)\jump{u}_{g}\jump{v}+\bbeta\cdot\ns\mean{u}_{g}\jump{v})\dif\by,
\end{split}
\end{align}
where the gradients are taken element wise. The bilinear form $\OpB(\cdot,\cdot,\mathfrak{T},g)$ will be used to approximate elliptic problems in $D$ with Dirichlet boundary condition $g$. This scheme is known as the Symmetric Weighted Interior Penalty (SWIP) scheme \cite{ESZ09}. The SWIP method is an improvement of the Interior Penalty scheme (IP) \cite{Arn82}, where the weights are defined as $\omega_{K,\sigma}=\omega_{T,\sigma}=1/2$. The use of diffusivity-dependent averages increases the robustness of the method for problems with strong diffusion discontinuities. The bilinear form defined in \cref{eq:opB} is mathematically equivalent to other formulations where $v\bbeta\cdot\nabla u$ or $\nabla\cdot(\bbeta u)v$ appear instead of $u\bbeta\cdot\nabla v$ (see \cite{ESZ09} and \cite[Section 4.6.2]{PiE12}). Our choice of formulation is convenient to express local conservation laws (see \cite[Section 2.2.3]{PiE12}).

\subsection{Local method algorithm}\label{sec:localg}
In this section we present the local scheme. In order to facilitate the comprehension of the method, we start with an informal description and then provide a pseudo-code for the algorithm. We denote $u_k$ the global solutions on $\Omega$ and $\huk$ the local solutions on $\Ok$, which are used to correct the global solutions.

Given a discretization $\mathfrak{T}_1=(\Omega,\M_1,\F_1)$ on $\Omega$ the local scheme computes a first approximate solution $u_1\in V(\mathfrak{T}_1)$ to \eqref{eq:weak}. The algorithm then performs the following steps for $k=2,\ldots,M$.
\begin{enumerate}[label=\roman*)]
	\item Given the current solution $u_{k-1}$, identify the region $\Omega_k$ where the error is large and define a new refined mesh $\Mk$ satisfying \cref{ass:mesh} by iterating the following steps.
	\begin{enumerate}[label=\alph*)]
		\item For each element $K\in\mathcal{M}_{k-1}$ compute an error indicator $\eta_{M,K}$ (defined in \eqref{eq:marketa}) and mark the local domain $\Omega_{k}$ using the fixed energy fraction marking strategy \cite[Section 4.2]{Dor96}. Hence, $\Omega_{k}$ is defined as the union of the elements with largest error indicator $\eta_{M,K}$ and it is such that the error committed inside of $\Omega_{k}$ is at least a prescribed fraction of the total error.
		\item Define the new mesh ${\M}_{k}$ by refining the elements $K\in\M_{k-1}$ with $K\subset\Omega_{k}$.
		\item Enlarge the local domain $\Omega_{k}$ defined at step a) by adding a one element wide boundary layer (i.e. in order to satisfy item 2b of \cref{ass:mesh}).
		\item Define the local mesh $\widehat{\mathcal{M}}_{k}$ by the elements of $\mathcal{M}_{k}$ inside of $\Omega_{k}$.
	\end{enumerate}
	\item Solve a local elliptic problem in $\Omega_k$ on the refined mesh $\hMk$ using $u_{k-1}$ as artificial Dirichlet boundary condition on $\partial\Ok\setminus\partial\Omega$. The solution is denoted $\huk\in V(\widehat{\mathfrak{T}}_k)$, where $\widehat{\mathfrak{T}}_k=(\Ok,\hMk,\hFk)$.
	\item The local solution $\huk$ is used to correct the previous solution $u_{k-1}$ inside of $\Ok$ and obtain the new global solution $u_k$.
\end{enumerate}
The pseudo-code of the local scheme is given in \cref{alg:local}, where $\chi_{\Omega\setminus\Ok}$ is the indicator function of $\Omega\setminus\Ok$ and $(\cdot,\cdot)_k$ is the inner product in $L^2(\Omega_k)$. The function $\text{LocalDomain}(u_k,\mathfrak{T}_k)$ used in \cref{alg:local} performs steps a)-d) of i). For purely diffusive problems, it is shown in \cite[Theorem 8.2]{Ros20} that \cref{alg:local} is equivalent to the LDGGD introduced in \cite{AbR19}, hence the scheme convergences for exact solutions $u\in H^1_0(\Omega)$.

\begin{algorithm}
	\caption{LocalScheme($\mathfrak{T}_1$)}
	\label{alg:local}
	\begin{algorithmic}
		\State Find $u_1\in V(\mathfrak{T}_1)$ solution to $\OpB(u_1,v_1,\mathfrak{T}_1,0)=(f,v_1)_1$ for all $v_1\in V(\mathfrak{T}_1)$.
		\For{$k=2,\ldots,M$}
		\State $(\mathfrak{T}_k,\widehat{\mathfrak{T}}_{k}) = \text{LocalDomain}(u_{k-1},\mathfrak{T}_{k-1})$.
		\State $g_k=u_{k-1}\chi_{\Omega\setminus\Ok}\in V(\mathfrak{T}_k)$.
		\State Find $\huk\in V(\widehat{\mathfrak{T}}_k)$ solution to $\OpB(\huk,v_k,\widehat{\mathfrak{T}}_k,g_k)=(f,v_k)_k$ for all $v_k\in V(\widehat{\mathfrak{T}}_k)$.
		\State $u_k=g_k+\huk\in V(\mathfrak{T}_k)$.
		\EndFor
	\end{algorithmic}
\end{algorithm}

\section{Error estimators via flux and potential reconstructions}\label{sec:err_flux}
The error estimators used to mark the local domains $\Omega_k$ and to provide error bounds on the numerical solution $u_k$ are introduced here.

In the framework of selfadjoint elliptic problems, the equilibrated fluxes method \cite{AiO93,BaW85} is a technique largely used to derive a posteriori error estimators free of undetermined constants and is based on the definition of local fluxes which satisfy a local conservation property. Since local fluxes and conservation properties are intrinsic to the discontinuous Galerkin formulation, this discretization is well suited for the equilibrated fluxes method \cite{Ain05,CoN08}. In \cite{ENV07,Kim07} the Raviart-Thomas-Nédélec space is used to build an $H_{\divop}(\Omega)$ conforming reconstruction $\bt_h$ of the discrete diffusive flux $-A\nabla u_h$. A diffusive flux $\bt_h$ with optimal divergence, in the sense that it coincides with the orthogonal projection of the right-hand side $f$ onto the discontinuous Galerkin space, is obtained. In \cite{ESV10} the authors extend this approach to convection-diffusion-reaction equations by defining an $H_{\divop}(\Omega)$ conforming convective flux $\bq_h$ approximating $\bbeta u_h$ and satisfying a conservation property.

We follow a similar strategy and define in the next section error estimators in function of diffusive and convective fluxes reconstructions $\btk,\bqk$ for the local scheme, as well as an $H^1_0(\Omega)$ conforming potential reconstruction $s_k$ of the solution $u_k$.

\subsection{Definition of the error estimators}\label{sec:errest}
The error estimators in function of the potential reconstruction $s_k$ approximating the solution $u_k$, the diffusive and convective fluxes $\btk$ and $\bqk$ approximating $-A\nabla u_k$ and $\bbeta u_k$, respectively, are defined in this section.

Following the iterative and local nature of our scheme, we define the diffusive and convective fluxes reconstructions as
\begin{equation}\label{eq:defflux}
\btk=\bt_{k-1}\chi_{\Omega\setminus\Ok}+\hbtk,\qquad \bqk=\bq_{k-1}\chi_{\Omega\setminus\Ok}+\hbqk,
\end{equation}
where $\bt_0=\bq_0=0$ and $\hbtk$, $\hbqk$ are $H_{\divop}(\Ok)$ conforming fluxes reconstructions of $-A\nabla \huk$, $\bbeta \huk$, respectively, and where $\huk$ is the local solution. To avoid any abuse of notation in \cref{eq:defflux}, we extended $\hbtk$, $\hbqk$ to zero outside of $\Ok$.
The fluxes reconstructions $\hbtk$, $\hbqk$ satisfy a local conservation property and are defined in \cref{sec:potflux}. We readily see that \cref{eq:defflux} allows for flux jumps at the subdomains boundaries, while giving enough freedom to define $\hbtk,\hbqk$ in a way that a conservation property is satisfied. The fluxes reconstructions are used to measure the non conformity of the numerical fluxes. In the same spirit we define a potential reconstruction $s_k\in H^1_0(\Omega)$ used to measure the non conformity of the numerical solution. It is defined recursively as
\begin{equation}\label{eq:defpot}
s_k = s_{k-1}\chi_{\Omega\setminus\Ok}+\hat s_k,
\end{equation}
where $s_0=0$ and $\hat s_k\in H^1(\Ok)$ is such that $s_k\in H^1_0(\Omega)$; similarly, we extend $\hat s_k$ to zero outside of $\Ok$. More details about the definitions of $\hbtk$, $\hbqk$ and $\hat s_k$ will be given in \cref{sec:potflux}, for the time being we will define the error estimators.

Let $K\in\Mk$, $v\in H^1(K)$, 
\begin{equation}\label{eq:defnBK}
\nB{v}_K^2=\nLddK{A^{1/2}\nabla v}^2+\nLdK{(\mu-\frac{1}{2}\nabla\cdot\bbeta)^{1/2}v}^2,
\end{equation}
where $\nLdK{{\cdot}}$ is the $L^2$-norm for scalar-valued functions in $K$ and $\nLddK{{\cdot}}$ the $L^2$-norm for vector-valued functions in $K$. The non conformity of the numerical solution $u_k$ is measured by the estimator
\begin{subequations}
	\begin{equation}\label{eq:etaNC}
	\eta_{NC,K}= \nB{u_k-s_k}_K.
	\end{equation}
	In the following, $m_K$, $\tilde m_K$, $m_\sigma$, $D_{t,K,\sigma}$, $c_{\bbeta,\mu,K}>0$ are some known constants which will be defined in \cref{sec:ctedef}.
	The residual estimator is
	\begin{equation}\label{eq:etaR}
	\eta_{R,K}= m_K \nLdK{f-\nabla\cdot\btk-\nabla\cdot\bqk-(\mu-\nabla\cdot\bbeta)u_k},
	\end{equation}
	which can be seen as the residual of \eqref{eq:weak} where we first replace $u$ by $u_k$, then $-A\nabla u_k$ by $\btk$, $\bbeta u_k$ by $\bqk$ and finally use the Green theorem. The error estimators defined in \cref{eq:etaDF,eq:etaC1,eq:etaC2,eq:etaU,eq:etaG1,eq:etaG2,eq:etatC1,eq:etatU} measure the error introduced by these substitutions and the error introduced when applying the Green theorem to $\btk,\bqk$, which are not in $H_{\divop}(\Omega)$.
	
	The diffusive flux estimator measures the difference between  $-A\nabla u_k$ and $\btk$. It is given by $\eta_{DF,K}=\min\lbrace \eta_{DF,K}^1,\eta_{DF,K}^2\rbrace$, where
	\begin{equation}\label{eq:etaDF}
	\begin{aligned}
	\eta_{DF,K}^1 &= \nLddK{A^{1/2}\nabla u_k+A^{-1/2}\btk},\\
	\eta_{DF,K}^2 &= m_K\nLdK{(\OpI-\pi_0)(\nabla\cdot(A\nabla u_k+\btk))}\\
	&\quad +\tilde{m}_K^{1/2}\sum_{\sigma\in \mathcal{F}_K}C_{t,K,\sigma}^{1/2}\nLds{(A\nabla u_k+\btk)\cdot\bm{n}_\sigma},
	\end{aligned}
	\end{equation}
	$\pi_0$ is the $L^2$-orthogonal projector onto $\mathbb{P}_0(K)$ and $\OpI$ is the identity operator. Let $\sigma\in\mathcal{F}_k$ and $\pi_{0,\sigma}$ be the $L^2$-orthogonal projector onto $\mathbb{P}_0(\sigma)$. The convection and upwinding estimators, that measure the difference between $\bbeta u_k$, $\bbeta s_k$ and $\bqk$, are defined by
	\begin{align}\label{eq:etaC1}
	\eta_{C,1,K}&= m_K\nLdK{(\OpI-\pi_0)(\nabla\cdot(\bqk-\bbeta s_k))},\\ \label{eq:etaC2}
	\eta_{C,2,K}&= \frac{1}{2}c_{\bbeta,\mu,K}^{-1/2}\nLdK{(\nabla\cdot\bbeta)(u_k-s_k))},\\ \label{eq:etatC1}
	\tilde \eta_{C,1,K}&= m_K\nLdK{(\OpI-\pi_0)(\nabla\cdot (\bqk-\bbeta u_k))},\\ \label{eq:etaU}
	\eta_{U,K} &= \sum_{\sigma\in\mathcal{F}_K}\chi_\sigma m_\sigma\nLds{\pi_{0,\sigma}\mean{\bqk-\bbeta s_k}\cdot \bm{n}_\sigma},\\ \label{eq:etatU}
	\tilde\eta_{U,K}&= \sum_{\sigma\in\mathcal{F}_K}\chi_\sigma m_\sigma\nLds{\pi_{0,\sigma}\mean{\bqk-\bbeta u_k}\cdot \bm{n}_\sigma},
	\end{align}
	where $\chi_\sigma=2$ if $\sigma\in\Fkb$ and $\chi_\sigma=1$ if $\sigma\in\Fki$.
	Finally, we introduce the jump estimators coming from the application of the Green theorem to $\btk$ and $\bqk$ (see \cref{lemma:boundBBA}). Those are defined by 
	\begin{align}\label{eq:etaG1}
	\eta_{\Gamma,1,K} &= \frac{1}{2}(|K|c_{\bbeta,\mu,K})^{-1/2}\sum_{\sigma\in\mathcal{F}_K\cap\Fki}\nLus{\pi_{0,\sigma}\jump{\bqk}\cdot\bm{n}_\sigma},\\ \label{eq:etaG2}
	\eta_{\Gamma,2,K} &= \frac{1}{2}\sum_{\sigma\in\mathcal{F}_K\cap\Fki} D_{t,K,\sigma}\nLds{\jump{\btk}\cdot \bm{n}_\sigma}.
	\end{align}
\end{subequations}
We end the section defining the marking error estimator $\eta_{M,K}$ used to mark $\Ok$ in the LocalDomain routine of \cref{alg:local}, let
\begin{equation}\label{eq:marketa}
\begin{aligned}
\eta_{M,K}&= \eta_{NC,K}+\eta_{R,K}+\eta_{DF,K}+\eta_{C,1,K}+\eta_{C,2,K}+\eta_{U,K}\\
&\quad +\eta_{\Gamma,1,K}+\eta_{\Gamma,2,K}+\tilde\eta_{C,1,K}+\tilde\eta_{U,K}.
\end{aligned}
\end{equation}

\subsection{Main results}\label{sec:thms}
We state here our main results related to the a posteriori analysis of the local scheme, in particular we will provide reliable error bounds on the numerical solution $u_k$ which are free of undetermined constants. We will also comment as to why we cannot prove the efficiency of the new estimator.

We start defining the norms for which we provide the error bounds, the same norms are used in \cite{ESV10}. The operator $\OpB$ defined in \eqref{eq:bform} can be written $\OpB=\OpB_S+\OpB_A$, where $\OpB_S$ and $\OpB_A$ are symmetric and skew-symmetric operators defined by
\begin{equation}\label{eq:bsba}
\begin{aligned}
\OpB_S(u,v)&= \int_\Omega (A\nabla u\cdot\nabla v+(\mu-\frac{1}{2}\nabla\cdot\bbeta)u v)\dif\bx,\\
\OpB_A(u,v)&=\int_{\Omega}(\bbeta\cdot\nabla u+\frac{1}{2}(\nabla\cdot\bbeta)u)v\dif\bx,
\end{aligned}
\end{equation}
for $u,v\in  H^1(\mathcal{M}_k)$. The energy norm is defined by the symmetric operator as
\begin{equation}
\nB{v}^2 = \OpB_S(v,v) = \nLdd{A^{1/2}\nabla v}^2+\nLd{(\mu-\frac{1}{2}\nabla\cdot\bbeta)^{1/2}v}^2,
\end{equation}
observe that $\nB{v}^2=\sum_{K\in\Mk}\nB{v}_K^2$, with $\nB{{\cdot}}_K$ as in \eqref{eq:defnBK}. Since the norm $\nB{{\cdot}}$ is defined by the symmetric operator, it is well suited to study problems with dominant diffusion or reaction. On the other hand, it is inappropriate for convection dominated problems since it lacks a term measuring the error along the velocity direction. For this kind of problems we use the augmented norm
\begin{equation}\label{eq:augnorm}
\nBp{v}=\nB{v}+\sup_{\substack{w\in H^1_0(\Omega)\\ \nB{w}=1}}(\OpB_A(v,w)+\OpB_J(v,w)),
\end{equation}
where
\begin{equation}
\OpB_J(v,w)=-\sum_{\sigma\in\Fki}\int_\sigma \jump{\bbeta v}\cdot\bm{n}_\sigma \mean{\pi_0 w}\dif\by
\end{equation}
is a term needed to sharpen the error bounds. The next two theorems give a bound on the error of the local scheme, measured in the energy or the augmented norm.
\begin{theorem}\label{thm:energynormbound}
	Let $u\in H^1_0(\Omega)$ be the solution to \eqref{eq:weak}, $u_k\in V(\mathfrak{T}_k)$ given by \cref{alg:local}, $s_k\in V(\mathfrak{T}_k)\cap H^1_0(\Omega)$ from \cref{eq:defpot,eq:defhsk} and $\btk,\bqk\in \mathbf{RTN}_\ar(\Mk)$ be defined by \cref{eq:defflux,eq:deflocflux}. Then, the error measured in the energy norm is bounded as
	\begin{equation}
	\nB{u-u_k}\leq \eta =  \left(\sum_{K\in\Mk}\eta_{NC,K}^2\right)^{1/2}+\left(\sum_{K\in\Mk}\eta_{1,K}^2\right)^{1/2},
	\end{equation}
	where $\eta_{1,K}=\eta_{R,K}+\eta_{DF,K}+\eta_{C,1,K}+\eta_{C,2,K}+\eta_{U,K}+\eta_{\Gamma,1,K}+\eta_{\Gamma,2,K}$.
\end{theorem}
\begin{theorem}\label{thm:augmentednormbound}
	Under the same assumptions of \cref{thm:energynormbound}, the error measured in the augmented norm is bounded as
	\begin{equation}
	\nBp{u-u_k}\leq \tilde\eta =  2\eta +\left(\sum_{K\in\Mk}\eta_{2,K}^2\right)^{1/2},
	\end{equation}
	with $\eta$ from \cref{thm:energynormbound} and $\eta_{2,K}=\eta_{R,K}+\eta_{DF,K}+\tilde\eta_{C,1,K}+\tilde\eta_{U,K}+\eta_{\Gamma,1,K}+\eta_{\Gamma,2,K}$.
\end{theorem}
The error estimators of \cref{thm:energynormbound,thm:augmentednormbound} are free of undetermined constants, indeed they depend on the numerical solution, the smallest eigenvalues of the diffusion tensor, on the essential minimum of $\mu-\frac{1}{2}\nabla\cdot\bbeta$, the mesh size and known geometric constants. In contrast, the error estimators are not efficient. The reason is that, compared to the true errors $\nB{u-u_k}$ and $\nBp{u-u_k}$, the error estimators $\eta_{\Gamma,1,K},\eta_{\Gamma,2,K}$ have a lower order of convergence. We illustrate this numerically in \cref{exp:conv}.
However, $\eta_{\Gamma,1,K},\eta_{\Gamma,2,K}$ are useful in practice: whenever they are small, then the error estimators are efficient. When they become large then they indicate that the error is not localized and one should switch to a nonlocal method. This is also illustrated numerically in \cref{exp:conv}.

\section{Potential and fluxes reconstructions, proofs of the main results}\label{sec:errbound}
In this section, we will define the potential, diffusion and advection reconstructions, define the geometric constants appearing in the error estimators defined in \cref{eq:etaNC,eq:etaR,eq:etaDF,eq:etaC1,eq:etaC2,eq:etaU,eq:etaG1,eq:etaG2,eq:etatC1,eq:etatU} and finally prove \cref{thm:energynormbound,thm:augmentednormbound}.

\subsection{Potential and fluxes reconstruction via the equilibrated flux method}\label{sec:potflux}
We define here the fluxes reconstructions $\hbtk$, $\hbqk$ of \eqref{eq:defflux} and the potential reconstruction $\hat s_k$ of \eqref{eq:defpot}. In what follows we assume that $\Mk$ does not have hanging nodes, i.e. we consider matching meshes, since it simplifies the analysis; however, in practice nonmatching meshes possessing hanging nodes can be employed (as in \cref{sec:num}). Roughly speaking, the next results are extended to nonmatching meshes by building matching submeshes and computing the error estimators on those submeshes, we refer to \cite[Appendix]{ESV10} for the details.

We start defining some broken Sobolev spaces and then the potential and fluxes reconstructions. For $k=1,\ldots,M$ let $\mathcal{G}_k=\{G_j\,|\, j=1,\ldots,k\}$, where $G_k=\Omega_k$ and 
\begin{equation}
G_j =\Omega_j\setminus\cup_{i=j+1}^{k}\overline{\Omega}_{i} \qquad \text{for }j=1,\ldots,k-1.
\end{equation}
In \cref{fig:Omegak,fig:Dk} we give an example of a sequence of domains $\Omega_k$ and the corresponding set $\mathcal{G}_k$.
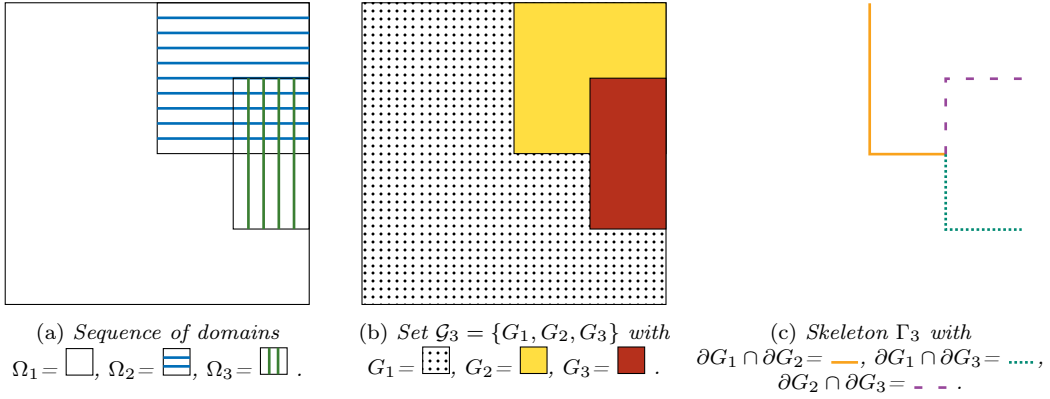
\begin{figure}
	\begin{center}
		\begin{subfigure}[t]{0.3\textwidth}
			\centering
			\captionsetup{justification=centering}
			\begin{tikzpicture}[scale=1]
			\draw (0,0) rectangle (4,4);
			\draw (2,2) rectangle (4,4);
			\draw[color=NavyBlue,line width=1pt] (2,2.2)--(4,2.2);
			\draw[color=NavyBlue,line width=1pt] (2,2.4)--(4,2.4);
			\draw[color=NavyBlue,line width=1pt] (2,2.6)--(4,2.6);
			\draw[color=NavyBlue,line width=1pt] (2,2.8)--(4,2.8);
			\draw[color=NavyBlue,line width=1pt] (2,3)--(4,3);
			\draw[color=NavyBlue,line width=1pt] (2,3.2)--(4,3.2);
			\draw[color=NavyBlue,line width=1pt] (2,3.4)--(4,3.4);
			\draw[color=NavyBlue,line width=1pt] (2,3.6)--(4,3.6);
			\draw[color=NavyBlue,line width=1pt] (2,3.8)--(4,3.8);
			\draw (3,1) rectangle (4,3);
			\draw[color=OliveGreen,line width=1pt] (3.2,1)--(3.2,3);
			\draw[color=OliveGreen,line width=1pt] (3.4,1)--(3.4,3);
			\draw[color=OliveGreen,line width=1pt] (3.6,1)--(3.6,3);
			\draw[color=OliveGreen,line width=1pt] (3.8,1)--(3.8,3);
			\end{tikzpicture}
			\caption{Sequence of domains\\$\Omega_1$= \tikz \draw (0,0) rectangle (10pt,10pt);, $\Omega_2$= \tikz{\draw(0,0) rectangle (10pt,10pt);\draw[color=NavyBlue,line width=1pt] (0,6.6pt)--(10pt,6.6pt);\draw[color=NavyBlue,line width=1pt] (0,3.3pt)--(10pt,3.3pt);}, $\Omega_3$= \tikz{\draw(0,0) rectangle (10pt,10pt);\draw[color=OliveGreen,line width=1pt] (6.6pt,0pt)--(6.6pt,10pt);\draw[color=OliveGreen,line width=1pt] (3.3pt,0pt)--(3.3pt,10pt);} .}
			\label{fig:Omegak}
		\end{subfigure}
		\begin{subfigure}[t]{0.3\textwidth}
			\centering
			\captionsetup{justification=centering}
			\begin{tikzpicture}[scale=1]
			\draw[pattern=dots] (0,0) rectangle (4,4);
			\draw[fill=Goldenrod] (2,2) rectangle (4,4);
			\draw[fill=BrickRed] (3,1) rectangle (4,3);
			\end{tikzpicture}
			\caption{Set $\mathcal{G}_3=\{G_1,G_2,G_3\}$ with\\ $G_1$= \tikz \draw[pattern=dots] (0,0) rectangle (10pt,10pt);, $G_2$= \tikz \draw[fill=Goldenrod] (0,0) rectangle (10pt,10pt);, $G_3$= \tikz \draw[fill = BrickRed] (0,0) rectangle (10pt,10pt); .}
			\label{fig:Dk}
		\end{subfigure}
		\begin{subfigure}[t]{0.3\textwidth}
			\centering
			\captionsetup{justification=centering}
			\begin{tikzpicture}[scale=1]
			\draw[draw=none] (0,0) rectangle (4,4);
			\draw[color=YellowOrange, line width=1pt, solid] (2,4)--(2,2)--(3,2);
			\draw[color=PineGreen, line width=1pt, densely dotted] (3,2)--(3,1)--(4,1);
			\draw[color=Purple, line width=1pt, loosely dashed] (3,2)--(3,3)--(4,3);
			\end{tikzpicture}
			\caption{Skeleton $\Gamma_3$ with \\$\partial G_1\cap\partial G_2$= \tikz\draw[color=YellowOrange, line width=1pt, solid] (0,0)--(10pt,0pt);, $\partial G_1\cap\partial G_3$= \tikz \draw[color=PineGreen, line width=1pt, densely dotted] (0,0) -- (10pt,0pt);,\\$\partial G_2\cap\partial G_3$= \tikz \draw[color=Purple, line width=1pt, loosely dashed] (0,0) -- (15pt,0pt);.}
			\label{fig:Gk}
		\end{subfigure}
	\end{center}
	\caption{Example of sequence of domains $\Omega_1,\Omega_2,\Omega_3$, set $\mathcal{G}_3$ and skeleton $\Gamma_3$.}
	\label{fig:illustrationDk}
\end{figure}
We define the broken spaces
\begin{align}
H_{\divop}(\mathcal{G}_k) &= \{\bv\in L^2(\Omega)^d\,:\, \bv|_G\in H_{\divop}(G)\text{ for all }G\in \mathcal{G}_k\},\\
H^1({\mathcal{M}}_k)&=\{v\in L^2(\Omega)\,:\,v|_K\in H^1(K)\text{ for all }K\in\mathcal{M}_k\},
\end{align}
the divergence and gradient operators in $H_{\divop}(\mathcal{G}_k)$ and $H^1(\Mk)$ are taken element wise.
We extend the jump operator $\jump{\cdot}_\sigma$ to the broken space $H^1(\Mk)$. We call $\Gamma_k$ the internal skeleton of $\mathcal{G}_k$, that is
\begin{equation}
\Gamma_k=\{\partial G_i\cap\partial G_j\,|\, G_i,G_j\in\mathcal{G}_k,\, i\neq j\},
\end{equation}
an example of $\Gamma_k$ is given in \cref{fig:Gk}.
For each $\gamma\in\Gamma_k$ we define $\mathcal{F}_\gamma = \{\sigma\in\Fki\,|\,\sigma\subset \gamma\}$ and set $\bm{n}_\gamma$, the normal to $\gamma$, as $\bm{n}_\gamma|_\sigma=\bm{n}_\sigma$. The jump $\jump{\cdot}_\gamma$ on $\gamma$ is defined by $\jump{\cdot}_\gamma|_\sigma=\jump{\cdot}_\sigma$.

In \cite{ESV10} the reconstructed fluxes live in $H_{\divop}(\Omega)$. For the local algorithm we need to build such fluxes using the recursive relation \eqref{eq:defflux}. This leads to fluxes having jumps across the boundaries of the subdomains, i.e. $\gamma\in\Gamma_k$, hence they lie in the broken space $H_{\divop}(\mathcal{G}_k)$. In the rest of this section we explain how to build fluxes which are in an approximation space of $H_{\divop}(\mathcal{G}_k)$ and satisfy a local conservation property. 
We start by introducing a broken version of the usual Raviart-Thomas-Nédélec spaces \cite{Ned80,RaT77}, which we define as
\begin{equation}\label{eq:RTN}
\mathbf{RTN}_\ar(\Mk):=\{\bm{v}_k\in H_{\divop}(\mathcal{G}_k)\,:\, \bm{v}_k|_K\in\mathbf{RTN}_\ar(K)\text{ for all }K\in\Mk\},
\end{equation}
where $\ar\in\{\ell-1,\ell\}$ and $\mathbf{RTN}_\ar(K)=\mathbb{P}_\ar(K)^d+\bx \mathbb{P}_{\ar}(K)$. In order to build functions in $\mathbf{RTN}_\ar(\Mk)$ we need a characterization of this space. 
Let $\bm{v}_k\in L^2(\Omega)^d$ such that $\bm{v}_k|_K\in\mathbf{RTN}_\ar(K)$ for each $K\in\Mk$, it is known that $\bm{v}_k\in H_{\divop}(\Omega)$ if and only if $\jump{\bm{v}_k}_\sigma\cdot\bm{n}_\sigma=0$ for all $\sigma\in\Fki$ (see \cite[Lemma 1.24]{PiE12}). Since we search for fluxes $\bm{v}_k$ in $H_{\divop}(\mathcal{G}_k)$, we relax this condition and allow $\jump{\bm{v}_k}_\gamma\cdot\bm{n}_\gamma\neq 0$ for $\gamma\in\Gamma_k$.

\begin{lemma}
	Let $\bm{v}_k\in L^2(\Omega)^d$ be such that $\bm{v}_k|_K\in\mathbf{RTN}_\ar(K)$ for each $K\in\Mk$, then $\bm{v}_k\in \mathbf{RTN}_\ar(\Mk)$ if and only if $\jump{\bm{v}_k}_\sigma\cdot\bm{n}_\sigma=0$ for all $\sigma\notin \cup_{\gamma\in\Gamma_k}\mathcal{F}_\gamma$.
\end{lemma} 
\begin{proof}
	Following the lines of \cite[Lemma 1.24]{PiE12}.
\end{proof}
The diffusive and convective fluxes $\btk,\bqk\in \mathbf{RTN}_\ar(\Mk)$ are defined recursively as in \eqref{eq:defflux}, where $\hbtk,\hbqk\in \mathbf{RTN}_\ar(\hMk)$, with
\begin{equation}
\mathbf{RTN}_\ar(\hMk):=\{\bm{v}_k\in H_{\divop}(\Omega_k)\,:\, \bm{v}_k\in\mathbf{RTN}_\ar(K)\text{ for all }K\in\hMk\},
\end{equation}
are given by the relations
\begin{subequations}\label{eq:deflocflux}
	\begin{equation}\label{eq:deflocflux1}
	\begin{aligned}
	\int_\sigma \hbtk\cdot\bm{n}_\sigma p_k\dif\by&= \int_\sigma (-\mean{A\nabla \huk}_\omega\cdot\bm{n}_\sigma+\eta_\sigma\frac{\gamma_\sigma}{h_\sigma}\jump{\huk}_{g_k})p_k\dif\by,\\
	\int_\sigma \hbqk\cdot\bm{n}_\sigma p_k\dif\by &= \int_\sigma (\bbeta\cdot\bm{n}_\sigma\mean{\huk}_{g_k}+\nu_\sigma\jump{\huk}_{g_k})p_k\dif\by
	\end{aligned}
	\end{equation}
	for all $\sigma\in\hFk$ and $p_k\in \mathbb{P}_\ar(\sigma)$ and
	\begin{equation}\label{eq:deflocflux2}
	\begin{aligned}
	\int_K \hbtk \cdot\hbrk\dif\bx &= -\int_K A\nabla\huk\cdot\hbrk\dif\bx+\sum_{\sigma\in\mathcal{F}_K}\int_\sigma\omega_{K,\sigma}\jump{\huk}_{g_k} A|_K\hbrk\cdot\bm{n}_\sigma\dif\by,\\
	\int_K\hbqk\cdot\hbrk\dif\bx &= \int_K \huk\bbeta\cdot\hbrk\dif\bx
	\end{aligned}
	\end{equation}
\end{subequations}
for all $K\in\hMk$ and $\hbrk\in\mathbb{P}_{\ar-1}(K)^d$. Since $\hbtk|_K\cdot\bm{n}_\sigma$, $\hbqk|_K\cdot\bm{n}_\sigma\in\mathbb{P}_\ar(\sigma)$ (see \cite[Proposition 3.2]{BrF91}) then \eqref{eq:deflocflux1} defines $\hbtk|_K\cdot\bm{n}_\sigma$, $\hbqk|_K\cdot\bm{n}_\sigma$ on $\sigma$. The remaining degrees of freedom are fixed by \eqref{eq:deflocflux2} \cite[Proposition 3.3]{BrF91}.
Thanks to \eqref{eq:deflocflux1} we have $\jump{\hbtk}\cdot\bm{n}_\sigma=0$ and $\jump{\hbqk}\cdot\bm{n}_\sigma=0$ for $\sigma\in\hFki$ and hence $\hbtk,\hbqk\in \mathbf{RTN}_\ar(\hMk)$. By construction it follows $\btk,\bqk\in \mathbf{RTN}_\ar(\Mk)$.

Let $K\in\Mk$ and $\pi_\ar$ be the $L^2$-orthogonal projector onto $\mathbb{P}_\ar(K)$, the following lemma states a local conservation property of the reconstructed fluxes. The proof follows the lines of \cite[Lemma 2.1]{ESV10}
\begin{lemma}\label{lemma:cons}
	Let $u_k\in V(\mathfrak{T}_k)$ be given by \cref{alg:local} and $\btk,\bqk\in H_{\divop}(\mathcal{G}_k)$ defined by \cref{eq:defflux,eq:deflocflux}. For all $K\in\Mk$ it holds
	\begin{equation}\label{eq:cons}
	(\nabla\cdot \btk+\nabla\cdot\bqk+\pi_\ar((\mu-\nabla\cdot\bbeta)u_k))|_K = \pi_\ar f|_K.
	\end{equation}
\end{lemma}
\begin{proof}
	Let $K\in\Mk$ and $j=\max\{i=1,\ldots,k\,:\, K\subset\Omega_j\}$, then $K\in\widehat{\mathcal{M}}_j$, $\btk|_K=\hbtj|_K$, $\bqk|_K=\hbqj|_K$ and $u_k|_K=\hat u_j|_K$. Let $v_j\in \mathbb{P}_{\ar}(K)$, with $v_j=0$ outside of $K$, by the Green theorem we have
	\begin{equation}\label{eq:greentq}
	\int_K (\nabla\cdot \hbtj+\nabla\cdot\hbqj)v_j\dif \bx = -\int_K (\hbtj+\hbqj)\cdot\nabla v_j\dif \bx +\sum_{\sigma\in\mathcal{F}_K}\int_\sigma v_j(\hbtj+\hbqj)\cdot\bm{n}_K\dif \by
	\end{equation}
	and using $\OpB(\hat u_j,v_j,\widehat{\mathfrak{T}}_j,g_j)=(f,v_j)_j$ it follows
	\begin{equation}
	\begin{aligned}
	\int_K f v_j \dif \bx &= \int_K (A\nabla \hat u_j\cdot \nabla v_j+(\mu-\nabla\cdot \bbeta)\hat u_j v_j-\hat{u}_j\bbeta\cdot \nabla v_j)\dif \bx\\
	&\quad -\sum_{\sigma\in\mathcal{F}_K}\int_\sigma(\jump{v_j}\mean{A\nabla \hat{u}_j}_{\omega}\cdot \bm{n}_\sigma+\jump{\hat{u}_j}_{g_j}\mean{A\nabla v_j}_{\omega}\cdot \bm{n}_\sigma)\dif \by\\
	&\quad +\sum_{\sigma\in\mathcal{F}_K}\int_\sigma ((\eta_\sigma\frac{\gamma_\sigma}{h_\sigma}+\nu_\sigma)\jump{\hat{u}_j}_{g_j}\jump{v_j}+\bbeta\cdot\bm{n}_\sigma\mean{\hat{u}_j}_{g_j}\jump{v_j})\dif \by.
	\end{aligned}
	\end{equation}
	Since $\mean{A\nabla v_j}_\omega =\omega_{K,\sigma}A|_K\nabla v_j$ and $\jump{v_j}\bm{n}_\sigma=v_j|_K\bm{n}_K$, using \cref{eq:deflocflux,eq:greentq}, we obtain
	\begin{equation}\label{eq:precons}
	\int_K f v_j \dif \bx = \int_K (\nabla\cdot \hbtj+\nabla\cdot\hbqj+(\mu-\nabla\cdot\bbeta)\hat u_j)v_j\dif \bx
	\end{equation}
	and the result follows from $\nabla\cdot\hbtj,\nabla\cdot\hbqj\in\mathbb{P}_\ar(K)$, $\btk|_K=\hat{\bt}_j|_K$, $\bqk|_K=\hat{\bq}_j|_K$ and $u_k|_K=\hat{u}_j|_K$. 
\end{proof}

In order to define the $H^1_0(\Omega)$ conforming approximation $s_k$ of $u_k$ we will need the so-called Oswald operator already considered in \cite{KaP03} for a posteriori estimates. Let $\mathfrak{T}=(D,\M,\F)$, $g\in C^0(\partial D)$ and consider $\mathcal{O}_{\mathfrak{T},g}:V(\mathfrak{T})\rightarrow V(\mathfrak{T})\cap H^1(D)$, for a function $v\in V(\mathfrak{T})$ the value of $\mathcal{O}_{\mathfrak{T},g} v$ is prescribed at the Lagrange interpolation nodes $p$ of the conforming finite element space $V(\mathfrak{T})\cap H^1(D)$. Let $p\in \overline{D}$ be a Lagrange node, if $p\notin \partial D$ we set
\begin{equation}
\mathcal{O}_{\mathfrak{T},g}v(p)=\frac{1}{\# \mathcal M_p}\sum_{K\in\mathcal{M}p}v|_K(p),
\end{equation}
where $\mathcal M_{p}=\{K\in\M\,:\,p\in\overline{K}\}$. If instead $p\in\partial D$ then $\mathcal{O}_{\mathfrak{T},g}v(p)=g(p)$, where $g$ is the Dirichlet condition on $\partial D$. The reconstructed potential $s_k\in V(\mathfrak{T}_k)\cap H^1_0(\Omega)$ is built as in \eqref{eq:defpot}, where
\begin{equation}\label{eq:defhsk}
\hat s_k = \mathcal{O}_{\widehat{\mathfrak{T}}_k,s_{k-1}} \huk.
\end{equation}

\subsection{Constants definition and preliminary results}\label{sec:ctedef}
Here we define the constants appearing in \cref{eq:etaNC,eq:etaR,eq:etaDF,eq:etaC1,eq:etaC2,eq:etaU,eq:etaG1,eq:etaG2,eq:etatC1,eq:etatU} and derive preliminary results needed to prove \cref{thm:energynormbound,thm:augmentednormbound}.

Let $K\in\Mk$ and $\sigma\in\mathcal{F}_K$, we recall that $|K|$ is the measure of $K$ and $|\sigma|$ the $d-1$ dimensional measure of $\sigma$. We denote by $c_{A,K}$ the minimal eigenvalue of $A|_K$. Next, we denote by $c_{\bbeta,\mu,K}$ the essential minimum of $\mu-\frac{1}{2}\nabla\cdot\bbeta\geq 0$ on $K$. 
In what follows we will assume that $\mu-\frac{1}{2}\nabla\cdot\bbeta>0$ a.e. in $\Omega$, hence $c_{\bbeta,\mu,K}>0$ for all $K\in\Mk$, and provide error estimators under this assumption. We explain in \cref{sec:altbounds} how to overcome this limitation slightly modifying the proofs and error estimators.

The cutoff functions $m_K,\tilde m_K$ and $m_\sigma$ are defined by 
\begin{subequations} \label{eq:cutoff}
	\begin{align} \label{eq:mK}
	m_K =& \min\{ C_p^{1/2}h_K c_{A,K}^{-1/2},c_{\bbeta,\mu,K}^{-1/2}\},\\ \label{eq:tmK}
	\tilde m_K=& \min\{ (C_p+C_p^{1/2})h_Kc_{A,K}^{-1}, h_K^{-1}c_{\bbeta,\mu,K}^{-1}+c_{\bbeta,\mu,K}^{-1/2}c_{A,K}^{-1/2}/2\},\\ \label{eq:ms}
	m_\sigma^2=& \min\lbrace \max_{K\in\mathcal{M}_\sigma}\{3d|\sigma|h_K^2|K|^{-1}c_{A,K}^{-1}\},\max_{K\in\mathcal{M}_\sigma}\{|\sigma||K|^{-1}c_{\bbeta,\mu,K}^{-1}\}\rbrace,
	\end{align}
\end{subequations}
where $C_p=1/\pi^2$ is an optimal Poincaré constant for convex domains \cite{PaW60}. Let $v\in H^1(\mathcal{M}_k)$, it holds
\begin{subequations}\label{eq:bounds}
	\begin{align} \label{eq:bounds1}
	\nLdK{v-\pi_0 v}&\leq m_K \nB{v}_K & \text{for all }& K\in\Mk,\\ \label{eq:bounds2}
	\nLds{v-\pi_0 v|_K}&\leq C_{t,K,\sigma}^{1/2}\tilde{m}_K^{1/2}\nB{v}_K & \text{for all }& \sigma\in \Fk \text{ and } K\in\mathcal{M}_\sigma,\\ \label{eq:bounds3}
	\nLds{\jump{\pi_0 v}}&\leq m_\sigma\sum_{K\in\mathcal{M}_\sigma}\nB{v}_K & \text{for all }& \sigma\in\Fk,
	\end{align}
\end{subequations}
where $\mathcal{M}_\sigma = \{K\in\Mk\,:\, \sigma\subset\partial K\}$ and $C_{t,K,\sigma}$ is the constant of the trace inequality
\begin{equation}\label{eq:trace}
\nLds{v|_K}^2\leq C_{t,K,\sigma}(h_K^{-1}\nLdK{v}^2+\nLdK{v}\nLddK{\nabla v}).
\end{equation}
It has been proved in \cite[Lemma 3.12]{Ste07} that for a simplex it holds $C_{t,K,\sigma}=|\sigma|h_K/|K|$. 

Let us briefly explain the role of constants \eqref{eq:cutoff} and how the bounds \eqref{eq:bounds} are obtained. We observe that for each bound in \eqref{eq:bounds} the cut off functions take the minimum between two possible values, allowing for robust error estimation in singularly perturbed regimes. For \eqref{eq:bounds1}, using the Poincaré inequality \cite[equation 3.2]{PaW60} we have
\begin{subequations}
	\begin{equation}\label{eq:bounds1a}
	\begin{aligned}
	\nLdK{v-\pi_0 v}&\leq C_p^{1/2} h_K \nLddK{\nabla v}\\
	& \leq C_p^{1/2}h_Kc_{A,K}^{-1/2}\nLddK{A^{1/2}\nabla v}\leq C_p^{1/2}h_Kc_{A,K}^{-1/2}\nB{v}_K.
	\end{aligned}
	\end{equation}
	Denoting $(\cdot,\cdot)_K$ the $L^2(K)$ inner product, it holds
	\begin{equation}
	\nLdK{v-\pi_0 v}^2=(v-\pi_0 v,v-\pi_0 v)_K=(v-\pi_0 v,v)_K\leq \nLdK{v-\pi_0 v}\nLdK{v},
	\end{equation}
	hence
	\begin{equation}\label{eq:bounds1b}
	\nLdK{v-\pi_0 v}\leq \nLdK{v} \leq c_{\bbeta,\mu,K}^{-1/2}\nLdK{(\mu-\frac{1}{2}\nabla\cdot\bbeta)^{1/2}v}\leq c_{\bbeta,\mu,K}^{-1/2}\nB{v}_K
	\end{equation}
\end{subequations}
and \eqref{eq:bounds1} follows. The choice between bounds \cref{eq:bounds1a,eq:bounds1b} depends on whether the problem is singularly perturbed or not. Bounds \eqref{eq:bounds2} and \eqref{eq:bounds3} are obtained similarly, see \cite[Lemma 4.2]{CFP09} and \cite[Lemma 4.5]{Voh08}. Finally, for $K\in\Mk$ and $\sigma\in \mathcal{F}_K$ we define
\begin{equation}\label{eq:Dk}
D_{t,K,\sigma}=\left(\frac{C_{t,K,\sigma}}{2 h_K c_{\bbeta,\mu,K}}\left(1+\sqrt{1+h_K^2\frac{c_{\bbeta,\mu,K}}{c_{A,K}}}\right)\right)^{1/2},
\end{equation}
which is used to bound $\nLds{v|_K}$ in terms of $\nB{v}_K$ in the next lemma.
\begin{lemma}\label{lemma:boundsigma}
	Let $v_k\in H^1(\Mk)$, for each $K\in\Mk$ and $\sigma\in \mathcal{F}_K$ it holds
	\begin{equation}
	\nLds{v_k|_K}\leq D_{t,K,\sigma} \nB{v_k}_K.
	\end{equation}
\end{lemma}
\begin{proof}
	Let $v_k\in H^1(\Mk)$ and $\epsilon>0$. Applying Hölder inequality to the trace inequality \cref{eq:trace} we get
	\begin{equation}
	\nLds{v_k|_K}^2 \leq C_{t,K,\sigma}((h_K^{-1}+\frac{1}{2\epsilon})\nLdK{v_k}^2+\frac{\epsilon}{2}\nLddK{\nabla v_k}^2).
	\end{equation}
	Hence, if there exists $D_{t,K,\sigma}>0$ independent of $v_k$ such that
	\begin{equation}\label{eq:Dkeps}
	\begin{aligned}
	C_{t,K,\sigma}((h_K^{-1}+\frac{1}{2\epsilon})\nLdK{v_k}^2+&\frac{\epsilon}{2}\nLddK{\nabla v_k}^2)\\
	& \leq D_{t,K,\sigma}^2 (c_{A,K}\nLddK{\nabla v_k}^2+c_{\bbeta,\mu,K}\nLdK{v_k}^2) 
	\end{aligned}
	\end{equation}
	then $\nLds{v_k|_K}^2\leq D_{t,K,\sigma}^2 \nB{v_k}^2_K$ and the result holds. Relation \eqref{eq:Dkeps} holds if
	\begin{equation}
	C_{t,K,\sigma}(h_K^{-1}+\frac{1}{2\epsilon})\leq D_{t,K,\sigma}^2c_{\bbeta,\mu,K}, \qquad\qquad C_{t,K,\sigma}\frac{\epsilon}{2} \leq D_{t,K,\sigma}^2c_{A,K}
	\end{equation}
	and hence $D_{t,K,\sigma}^2=\max\{C_{t,K,\sigma}(h_K^{-1}+\frac{1}{2\epsilon})c_{\bbeta,\mu,K}^{-1},C_{t,K,\sigma}\frac{\epsilon}{2}c_{A,K}^{-1}\}$.
	Taking $\epsilon$ such that the maximum is minimized we get $D_{t,K,\sigma}$ as in \cref{eq:Dk}.
\end{proof}
The proof of the following Lemma is inspired from \cite[Theorem 3.1]{ESV10}, the main difference is that we take into account the weaker regularity of the reconstructed fluxes. 
\begin{lemma}\label{lemma:boundBBA}
	Let $u\in H^1_0(\Omega)$ be the solution to \eqref{eq:weak}, $u_k\in V(\mathfrak{T}_k)$ given by \cref{alg:local}, $s_k\in H^1_0(\Omega)$ from \cref{eq:defpot,eq:defhsk}, $\btk,\bqk\in H_{\divop}(\mathcal{G}_k)$ defined by \cref{eq:defflux,eq:deflocflux} and $v\in H^1_0(\Omega)$. Then
	\begin{equation}
	|\OpB(u -u_k ,v)+\OpB_A(u_k-s_k,v)| \leq  \left(\sum_{K\in\Mk}\eta_{1,K}^2\right)^{1/2}\nB{v},
	\end{equation}
	with $\eta_{1,K}=\eta_{R,K}+\eta_{DF,K}+\eta_{C,1,K}+\eta_{C,2,K}+\eta_{U,K}+\eta_{\Gamma,1,K}+\eta_{\Gamma,2,K}$.
\end{lemma}
\begin{proof}
	Since $u$ satisfies \eqref{eq:weak}, using the definition of $\OpB$ and $\OpB_A$
	\begin{align}
	\OpB(u-u_k,v)+\OpB_A(u_k-s_k,v) 
	&= \int_\Omega (f-(\mu-\nabla\cdot\bbeta)u_k)v\dif \bx -\int_\Omega A\nabla u_k\cdot \nabla v\dif \bx\\
	&\quad -\int_\Omega \frac{1}{2}(\nabla\cdot\bbeta)(u_k-s_k)v\dif \bx -\int_\Omega \nabla\cdot(\bbeta s_k)v\dif \bx.
	\end{align}
	Using $v \btk\in H_{\divop}(\mathcal{G}_k)$, from the divergence theorem we have
	\begin{align}
	\int_\Omega (v\nabla\cdot \btk +\nabla v\cdot\btk)\dif \bx &= \sum_{G\in\mathcal{G}_k}\int_{G}\nabla\cdot(v\btk)\dif \bx =\sum_{G\in\mathcal{G}_k}\int_{\partial G} v\btk\cdot\bm{n}_{\partial G}\dif \by \\
	&=\sum_{\gamma\in\Gamma_k}\int_\gamma \jump{v \btk}\cdot \bm{n}_\gamma\dif \by = \sum_{\gamma\in\Gamma_k}\int_\gamma \jump{\btk}\cdot \bm{n}_\gamma v\dif \by
	\end{align}
	and hence
	\begin{equation}\label{eq:integrBBA}
	\begin{aligned}
	\OpB(u-u_k,v)+\OpB_A(u_k- s_k ,v)&=\int_\Omega (f-\nabla\cdot\btk-\nabla\cdot\bqk-(\mu-\nabla\cdot\bbeta)u_k)v\dif \bx \\
	&\quad -\int_\Omega \frac{1}{2}(\nabla\cdot\bbeta)(u_k-s_k)v\dif \bx +\int_\Omega \nabla\cdot(\bqk-\bbeta s_k)v\dif \bx\\
	&\quad -\int_\Omega (A\nabla u_k+\btk)\cdot \nabla v\dif \bx +\sum_{\gamma\in\Gamma_k}\int_\gamma \jump{\btk }\cdot\bm{n}_\gamma v\dif \by.
	\end{aligned}
	\end{equation}
	From \cref{lemma:cons} we deduce
	\begin{subequations}\label{eq:boundsBBAterms}
		\begin{equation}\label{eq:boundsBBAterm0}
		\begin{aligned}
		&\left|\int_\Omega (f-\nabla\cdot\btk-\nabla\cdot\bqk-(\mu-\nabla\cdot\bbeta)u_k)v\dif \bx\right| \\
		&\qquad\qquad\qquad\qquad = \left|\int_\Omega (f-\nabla\cdot\btk-\nabla\cdot\bqk-(\mu-\nabla\cdot\bbeta)u_k)(v-\pi_0 v)\dif \bx\right| \\
		&\qquad\qquad\qquad\qquad \leq \sum_{K\in\Mk} \eta_{R,K}\nB{v}_K.
		\end{aligned}
		\end{equation}
		Similarly, we get
		\begin{equation}\label{eq:boundsBBAterms1}
		\begin{aligned}
		\left| \int_\Omega (A\nabla u_k+\btk)\cdot \nabla v\dif \bx\right|&\leq  \sum_{K\in\Mk}\eta_{DF,K}\nB{v}_K,\\
		\left| \int_\Omega \frac{1}{2}(\nabla\cdot\bbeta)(u_k-s_k)v\dif \bx\right|&\leq  \sum_{K\in\Mk} \eta_{C,2,K}\nB{v}_K.
		\end{aligned}
		\end{equation}
		Since $\jump{\btk}_\sigma=0$ for $\sigma\in \Fki\setminus\cup_{\gamma\in\Gamma_k}\mathcal{F}_\gamma$, it holds
		\begin{equation}
		\sum_{\gamma\in\Gamma_k}\int_\gamma \jump{\btk}\cdot\bm{n}_\gamma v\dif \by = \sum_{\sigma\in\Fki}\int_\sigma\jump{\btk}\cdot\bm{n}_\sigma v\dif \by = \frac{1}{2}\sum_{K\in\Mk}\sum_{\sigma\in\mathcal{F}_K\cap\Fki}\int_\sigma \jump{\btk}\cdot\bm{n}_\sigma v\dif \by.
		\end{equation}
		Using \cref{lemma:boundsigma} we obtain
		\begin{equation}\label{eq:boundsBBAterms2}
		\begin{aligned}
		\left|\sum_{\gamma\in\Gamma_k}\int_\gamma \jump{\btk}\cdot\bm{n}_\gamma v\dif \by \right| &\leq  \frac{1}{2}\sum_{K\in\Mk}\sum_{\sigma\in\mathcal{F}_K\cap\Fki}\nLds{\jump{\btk}\cdot\bm{n}_\sigma}\nLds{v} \\
		&\leq   \sum_{K\in\Mk}\eta_{\Gamma,2,K}\nB{v}_K.
		\end{aligned}
		\end{equation}
		It remains to estimate $\int_\Omega \nabla\cdot(\bqk-\bbeta s_k)v\dif \bx$. For that, we use
		\begin{align}
		\int_\Omega \nabla\cdot(\bqk-\bbeta s_k)v\dif \bx 
		=& \sum_{K\in\Mk}\int_K (\OpI-\pi_0)\nabla\cdot(\bqk-\bbeta s_k)(v-\pi_0 v)\dif \bx \\
		&+\sum_{K\in\Mk} \sum_{\sigma\in\mathcal{F}_K}\int_\sigma (\bqk-\bbeta s_k)\cdot \bm{n}_K \pi_0 v\dif \by
		\end{align}
		and from \cref{eq:bounds1} we get
		\begin{align}\label{eq:boundsBBAterms3}
		\left|\sum_{K\in\Mk}\int_K (\OpI-\pi_0)\nabla\cdot(\bqk-\bbeta s_k)(v-\pi_0 v)\dif \bx \right|\leq \sum_{K\in\Mk}\eta_{C,1,K}\nB{v}_K.
		\end{align}
		For the second term we write
		\begin{align}
		&\sum_{K\in\Mk} \sum_{\sigma\in\mathcal{F}_K}\int_\sigma (\bqk-\bbeta s_k)\cdot \bm{n}_K \pi_0 v\dif \by= \sum_{\sigma\in\mathcal{F}_k}\int_\sigma \jump{\pi_{0,\sigma}(\bqk-\bbeta s_k)\pi_0 v}\cdot \bm{n}_\sigma\dif \by\\
		&=\sum_{\sigma\in\Fki}\int_\sigma \mean{\pi_0 v}\jump{\pi_{0,\sigma}(\bqk-\bbeta s_k)}\cdot \bm{n}_\sigma+\jump{\pi_0 v}\mean{\pi_{0,\sigma}(\bqk-\bbeta s_k)}\cdot \bm{n}_\sigma\dif \by\\
		&\quad +\sum_{\sigma\in \Fkb}\int_\sigma\pi_0 v\, \pi_{0,\sigma}(\bqk-\bbeta s_k)\cdot\bm{n}_\sigma \dif \by = \operatorname{I}+\operatorname{II}+\operatorname{III}
		\end{align}
		and we easily obtain, since $\jump{\bbeta s_k}=0$,
		\begin{equation}
		\operatorname{I} = \frac{1}{2}\sum_{K\in\Mk}\sum_{\sigma\in\mathcal{F}_K\cap\Fki}\int_\sigma \pi_0 v|_K \jump{\pi_{0,\sigma}\bqk}\cdot\bm{n}_\sigma\dif \by.
		\end{equation}
		Using $|\pi_0 v|_K| = |K|^{-1/2}\nLdK{\pi_0 v}\leq |K|^{-1/2}\nLdK{v}\leq (|K|c_{\bbeta,\mu,K})^{-1/2}\nB{v}_K$ we get
		\begin{equation}\label{eq:boundsBBAterms4}
		\operatorname{I} \leq \frac{1}{2}\sum_{K\in\Mk}\sum_{\sigma\in\mathcal{F}_K\cap\Fki}(|K|c_{\bbeta,\mu,K})^{-1/2}\nLus{\jump{\pi_{0,\sigma}\bqk}\cdot\bm{n}_\sigma}\nB{v}_K= \sum_{K\in\Mk}\eta_{\Gamma,1,K}\nB{v}_K.
		\end{equation}
		Let $\mathcal{M}_\sigma=\{K\in\Mk\,:\, \sigma\subset\partial K\}$, using \eqref{eq:bounds3} for the second term we have
		\begin{align}
		\operatorname{II} & \leq \sum_{\sigma\in\Fki} m_\sigma\nLds{\pi_{0,\sigma}\mean{\bqk-\bbeta s_k}\cdot\bm{n}_\sigma}\sum_{K\in \mathcal{M}_\sigma}\nB{v}_{K}\\
		&= \sum_{K\in\Mk} \sum_{\sigma\in\mathcal{F}_K\cap\Fki} m_\sigma\nLds{\pi_{0,\sigma}\mean{\bqk-\bbeta s_k}\cdot\bm{n}_\sigma}\nB{v}_K.
		\end{align}
		For the last term we similarly obtain
		\begin{equation}
		\operatorname{III} \leq \sum_{K\in\Mk} \sum_{\sigma\in\mathcal{F}_K\cap\Fkb} m_\sigma\nLds{\pi_{0,\sigma}(\bqk-\bbeta s_k)\cdot \bm{n}_\sigma}\nB{v}_K
		\end{equation}
		and hence
		\begin{equation}\label{eq:boundsBBAterms5}
		\operatorname{II}+\operatorname{III} \leq \sum_{K\in\Mk}\sum_{\sigma\in\mathcal{F}_K}\chi_\sigma m_\sigma\nLds{\pi_{0,\sigma}\mean{\bqk-\bbeta s_k}\cdot \bm{n}_\sigma} \nB{v}_K= \sum_{K\in\Mk}\eta_{U,K}\nB{v}_K,
		\end{equation}
	\end{subequations}
	where $\chi_\sigma=2$ if $\sigma\in\Fkb$ and $\chi_\sigma=1$ if $\sigma\in\Fki$. Plugging relations \cref{eq:boundsBBAterm0,eq:boundsBBAterms1,eq:boundsBBAterms2,eq:boundsBBAterms3,eq:boundsBBAterms4,eq:boundsBBAterms5} into \eqref{eq:integrBBA} we get the result.
\end{proof}

In \cref{lemma:boundBBA} we use \cref{lemma:cons} to deduce that
\begin{equation}\label{eq:weakcons}
\int_K (\nabla\cdot \btk+\nabla\cdot\bqk+(\mu-\nabla\cdot\bbeta)u_k) \dif \bx = \int_K f \dif \bx
\end{equation}
and hence \eqref{eq:boundsBBAterm0}. However, when the mesh has hanging nodes inside of the local domains \cref{lemma:cons} is not valid. Indeed, if $\hMk$ has hanging nodes, the fluxes $\hbtk,\hbqk$ must be constructed on a matching (free of hanging nodes) submesh $\overline{\mathcal{M}}_k$ of $\hMk$, otherwise they may fail to be in $H_{\divop}(\Omega_k)$. The constructed fluxes will satisfy relation \cref{eq:precons}, but since $\nabla\cdot\hbtk,\nabla\cdot\hbqk\in \mathbb{P}_\ar(K')$ for $K'\in\overline{\mathcal{M}}_k$ and $\overline{\mathcal{M}}_k$ is finer than $\hMk$, then we cannot conclude as we did in \cref{lemma:cons}. Nonetheless, \cref{eq:precons} still implies \cref{eq:weakcons}, which is enough to prove \cref{lemma:boundBBA}.

\subsection{Proof of the theorems}\label{sec:proofs}
Here we prove \cref{thm:energynormbound,thm:augmentednormbound}. We will consider $\OpB:H^1_0(\Omega)\times H^1_0(\Omega)\rightarrow\Rb$ defined in \eqref{eq:bform} for functions in $ H^1(\Mk)$.
\begin{proof}[Proof of \cref{thm:energynormbound}]
	It has been proved in \cite[Lemma 3.1]{Ern08} that for any $u_k\in V(\mathfrak{T}_k)$ and $u,s\in H^1_0(\Omega)$ it holds
	\begin{equation}
	\nB{u-u_k}\leq \nB{u_k-s}+|\OpB(u-u_k,v)+\OpB_A(u_k-s,v)|,
	\end{equation}
	with $v=(u-s)/\nB{u-s}$. Choosing $u$ as the exact solution to \cref{eq:weak}, $u_k$ given by \cref{alg:local}, $s=s_k$ from \cref{eq:defpot} and using \cref{lemma:boundBBA} gives the result.
\end{proof}

\begin{proof}[Proof of \cref{thm:augmentednormbound}]
	Since $u\in H^1_0(\Omega)$ it holds $\OpB_J(u,w)=0$ for all $w\in H^1_0(\Omega)$, using $\OpB_A\leq\OpB+|\OpB_S|$ we get
	\begin{equation}
	\nBp{u-u_k}\leq 2\nB{u-u_k}+\sup_{\substack{w\in H^1_0(\Omega)\\ \nB{w}=1}}(\OpB(u-u_k,w)-\OpB_J(u_k,w)).
	\end{equation}
	To conclude the proof we show that
	\begin{equation}\label{eq:supBBD}
	\sup_{\substack{w\in H^1_0(\Omega)\\ \nB{w}=1}}(\OpB(u-u_k,w)-\OpB_J(u_k,w))\leq \left(\sum_{K\in\Mk}\eta_{2,K}^2\right)^{1/2}.
	\end{equation}
	Following \cref{lemma:boundBBA}, we easily get
	\begin{multline}
	\OpB(u-u_k,w)-\OpB_J(u_k,w) \leq  \sum_{K\in\Mk}(\eta_{R,K}+\eta_{DF,K}+\tilde\eta_{C,1,K}+\eta_{\Gamma,2,K})\nB{w}_K\\
	+\sum_{K\in\Mk}\sum_{\sigma\in\mathcal{F}_K}\int_\sigma\pi_0 w (\bqk-\bbeta u_k)\cdot\bm{n}_K\dif \by-\OpB_J(u_k,w).
	\end{multline}
	The two last terms satisfy
	\begin{align}
	&\sum_{\sigma\in\Fk}\int_\sigma\jump{\pi_0 w(\bqk-\bbeta u_k)}\cdot \bm{n}_\sigma\dif \by-\OpB_J(u_k,w) \\
	&= \sum_{\sigma\in\Fk}\chi_\sigma\int_\sigma \jump{\pi_0 w}\pi_{0,\sigma}\mean{\bqk-\bbeta u_k}\cdot\bm{n}_\sigma\dif \by +\sum_{\sigma\in\Fki}\int_\sigma \mean{\pi_0 w}\jump{\pi_{0,\sigma}\bqk}\cdot\ns\dif \by \\
	&\leq\sum_{K\in\Mk}(\tilde\eta_{U,K}+\eta_{\Gamma,1,K})\nB{w}_K,
	\end{align}
	where in the last step we followed again \cref{lemma:boundBBA}.
\end{proof}

\subsection{Alternative error bounds}\label{sec:altbounds}
Our aim here is to explain how to avoid the assumption $c_{\bbeta,\mu,K}>0$ for all $K\in\Mk$ made in \cref{sec:errest,sec:ctedef}. This assumption is needed to define $\eta_{\Gamma,1,K}$, $\eta_{\Gamma,2,K}$ but can be avoided if \cref{eq:boundsBBAterms2,eq:boundsBBAterms4} are estimated differently. For \cref{eq:boundsBBAterms2}, using the trace inequality \cref{eq:trace} we get
\begin{equation}
\begin{aligned}
\left|\sum_{\gamma\in\Gamma_k}\int_\gamma \jump{\btk}\cdot\bm{n}_\gamma v\dif \by \right| &\leq  \frac{1}{2}\sum_{K\in\Mk}\sum_{\sigma\in\mathcal{F}_K\cap\Fki}\nLds{\jump{\btk}\cdot\bm{n}_\sigma}\nLds{v|_K} \\
&\leq \sum_{K\in\Mk}\tilde\eta_{\Gamma,2,K}(\nLdK{v}^2+h_K\nLdK{v}\nLddK{\nabla v})^{1/2},
\end{aligned}
\end{equation}
where
\begin{equation}
\tilde \eta_{\Gamma,2,K} = \frac{1}{2}\sum_{\sigma\in\mathcal{F}_K\cap\Fki}h_K^{-1/2}C_{t,K,\sigma}^{1/2}\nLds{\jump{\btk}\cdot\bm{n}_\sigma}.
\end{equation}
Setting $\tilde\eta_{\Gamma,2}^2=\sum_{K\in\Mk}\tilde \eta_{\Gamma,2,K}^2$, it yields
\begin{align}
\left|\sum_{\gamma\in\Gamma_k}\int_\gamma \jump{\btk}\cdot\bm{n}_\gamma v\dif \by \right| &\leq \tilde \eta_{\Gamma,2}\left(\sum_{K\in\Mk} \nLdK{v}^2+h_K\nLdK{v}\nLddK{\nabla v}\right)^{1/2}\\
&\leq \tilde \eta_{\Gamma,2} \left(\nLd{v}^2+h_{\Mk}\nLd{v}\nLdd{\nabla v}\right)^{1/2}.
\end{align}
Using the Poincaré inequality $\nLd{v}\leq d_\Omega\nLdd{\nabla v}$, where $d_\Omega$ is the diameter of $\Omega$, we get
\begin{equation}
\left|\sum_{\gamma\in\Gamma_k}\int_\gamma \jump{\btk}\cdot\bm{n}_\gamma v\dif \by \right| \leq \tilde \eta_{\Gamma,2} \left(d_\Omega^2+h_{\Mk}d_\Omega\right)^{1/2}\nLdd{\nabla v}\leq \tilde \eta_{\Gamma,2} c_A^{-1/2} \left(d_\Omega^2+h_{\Mk}d_\Omega\right)^{1/2}\nB{v},
\end{equation}
where $c_A$ is the minimal eigenvalue of $A(\bx)$ over $\Omega$. The same procedure can be used to replace \cref{eq:boundsBBAterms4} by a relation avoiding the term $c_{\bbeta,\mu,K}^{-1/2}$. The new bounds can be used to modify the results of \cref{thm:energynormbound,thm:augmentednormbound} and obtain error estimators when $\mu-\frac{1}{2}\nabla\cdot\bbeta>0$ is not satisfied.

\section{Numerical Experiments}\label{sec:num}
In order to study the properties and illustrate the performance of the local scheme we consider here several numerical examples.
First, in \cref{exp:conv}, we look at the convergence rates of the error estimators, focusing on the errors introduced by solving only local problems. Considering a local and a nonlocal problem, we also compare the size of the new error estimators $\eta_{\Gamma,1}$ and $\eta_{\Gamma,2}$ against the classical terms. We emphasize that we do not use the automatic subdomains' identification algorithm for this example, as the subdomains are fixed beforehand.
We also perform in \cref{exp:corner} an experiment for a smooth problem, where the errors are not localized, illustrating the role of $\eta_{\Gamma,1}$ and $\eta_{\Gamma,2}$. To do so, we also compare the local scheme against a classical adaptive method, where after each mesh refinement the problem is solved again on the whole domain. The classical method we refer to is given by \cref{alg:classical}.
Second, we investigate the efficiency of the new local algorithm for non smooth problems in \cref{exp:bndlayer_sym,exp:bndlayer_notsym}. For such examples, that are the target of our method, the local scheme performs better than the classical one. We conclude in \cref{exp:nonlin} with a nonlinear problem, where \cref{thm:energynormbound,thm:augmentednormbound} do not apply but \cref{alg:local} can nevertheless be employed in conjunction with a Newton scheme.

\begin{algorithm}[!tbhp]
	\caption{ClassicalScheme($\mathfrak{T}_1$)}
	\label{alg:classical}
	\begin{algorithmic}
		\State Find $\overline{u}_1\in V(\mathfrak{T}_1)$ solution to $\OpB(\overline{u}_1,v_1,\mathfrak{T}_1,0)=(f,v_1)_1$ for all $v_1\in V(\mathfrak{T}_1)$.
		\For{$k=2,\ldots,M$}
		\State $(\mathfrak{T}_k,\widehat{\mathfrak{T}}_{k}) = \text{LocalDomain}(\overline{u}_{k-1},\mathfrak{T}_{k-1})$.
		\State Find $\overline{u}_k\in V(\mathfrak{T}_k)$ solution to $\OpB(\overline{u}_k,v_k,\mathfrak{T}_k,0)=(f,v_k)_1$ for all $v_k\in V(\mathfrak{T}_k)$.
		\EndFor
	\end{algorithmic}
\end{algorithm}

In all the experiments we use $\mathbb P_1$ elements ($\ell=1$ in \eqref{eq:defVT}) on a simplicial mesh with penalization parameter $\eta_\sigma=10$, the diffusive and convective fluxes $\btk,\bqk$ are computed with $\ar=0$ (see \eqref{eq:RTN}). Furthermore, $\bbeta$ is always such that $\nabla\cdot\bbeta=0$. These choices give $\eta_{C,1,K}=\eta_{C,2,K}=\tilde\eta_{C,1,K}=0$. For an estimator $\eta_{*,K}$ we define $\eta_{*}^2=\sum_{K\in\M_k}\eta_{*,K}^2$.
Similarly to \cite{ESV10}, if $A=\varepsilon I_2$ and $\bbeta$ is constant then for $v_k\in H^1(\M_k)$ the augmented norm is well estimated by 
\begin{align}
\nBp{v_k}\leq \nB{v_k}_{\oplus'}&= \nB{v_k}+\varepsilon^{-1/2}\Vert\bbeta\Vert_2\nLd{v_k}\\
&\quad +\frac{1}{2}\left(\sum_{K\in\M_k}\left(\sum_{\sigma\in\mathcal{F}_K\cap\Fki}\tilde m_K^{1/2} C_{t,K,\sigma}^{1/2}\nLds{\jump{v_k}\bbeta\cdot\ns}\right)^2\right)^{1/2}.
\end{align}
Hence, in the numerical experiments we consider the computable norm $\nB{\cdot}_{\oplus'}$. The effectivity indexes of the error estimators $\eta$ and $\tilde \eta$ from \cref{thm:energynormbound,thm:augmentednormbound} are defined as
\begin{equation}\label{eq:effind}
\frac{\eta}{\nB{u-u_k}} \qquad\text{and}\qquad  \frac{\tilde\eta}{\nB{u-u_k}_{\oplus'}},
\end{equation}
respectively. For the solution $\overline u_k$ of the classical algorithm we use the error estimators $\eta$ and $\tilde \eta$ from \cite{ESV10}. They are equivalent to the estimators presented in this paper except that for $\overline u_k$ we have $\eta_{\Gamma,1,K}=\eta_{\Gamma,2,K}=0$, as in this case the reconstructed fluxes are in $H_{\divop}(\Omega)$. The effectivity indexes for $\overline u_k$ are as in \eqref{eq:effind} but with $u_k$ replaced by $\overline u_k$. The numerical experiments have been performed with the help of the C++ library \texttt{libMesh} \cite{KPS06}.

\subsection{Problem shifting from localized to nonlocalized errors}\label{exp:conv}
We investigate an example in two different locality regimes. First, the errors are confined in a small region and then they are distributed in the whole domain. We will study the effects of this transition on the size of the new error estimators $\eta_{\Gamma,1}$ and $\eta_{\Gamma,2}$.

We solve \eqref{eq:elliptic} in $\Omega=[0,1]\times [0,1]$ with $A=I_2$, $\bbeta=-(1,1)^\top$ and $\mu=1$. The force term $f$ is chosen so that the exact solution reads
\begin{equation}\label{eq:solsmooth}
	u(\bx)=e^{-\kappa ||\bx||_2}\left( x_1-\frac{1-e^{-\kappa x_1}}{1-e^{-\kappa}}\right)\left(x_2-\frac{1-e^{-\kappa x_2}}{1-e^{-\kappa}} \right),
\end{equation}
with $\kappa=100$ or $\kappa=10$. When $\kappa=100$ the solution has a narrow peak and the errors are localized around that region, when $\kappa=10$ the solution is smoother and the errors are distributed in the whole domain. See \cref{fig:sol_conv_100,fig:sol_conv_10}.

First, we investigate the convergence rate of the error estimators and then we comment on the size of the new error estimators $\eta_{\Gamma,1}$, $\eta_{\Gamma,2}$ when the errors are localized or not, i.e. when $\kappa=100$ or $\kappa=10$.
We define two domains $\Omega_1,\Omega_2$ as follows: $\Omega_1=\Omega$ and $\bx\in\Omega_2$ if $\Vert\bx\Vert_\infty\leq 1/2$, see \cref{fig:domains_priori}. 
\begin{figure}
	\begin{center}
		\begin{subfigure}[t]{\subfigsizesm\textwidth}
			\centering
			\includegraphics[trim=4cm 3cm 2.3cm 6.2cm, clip, width=\textwidth]{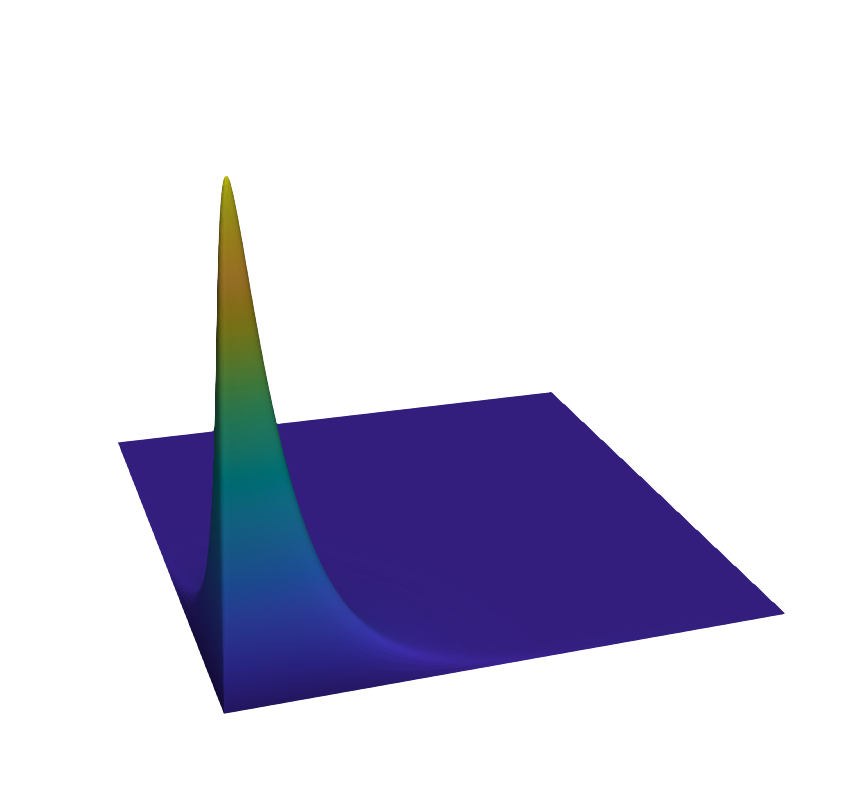}
			\caption{$u(\bx)$ for $\kappa=100$.}
			\label{fig:sol_conv_100}
		\end{subfigure}
		\begin{subfigure}[t]{\subfigsizesm\textwidth}
			\centering
			\includegraphics[trim=4cm 3cm 2.3cm 6.2cm, clip, width=\textwidth]{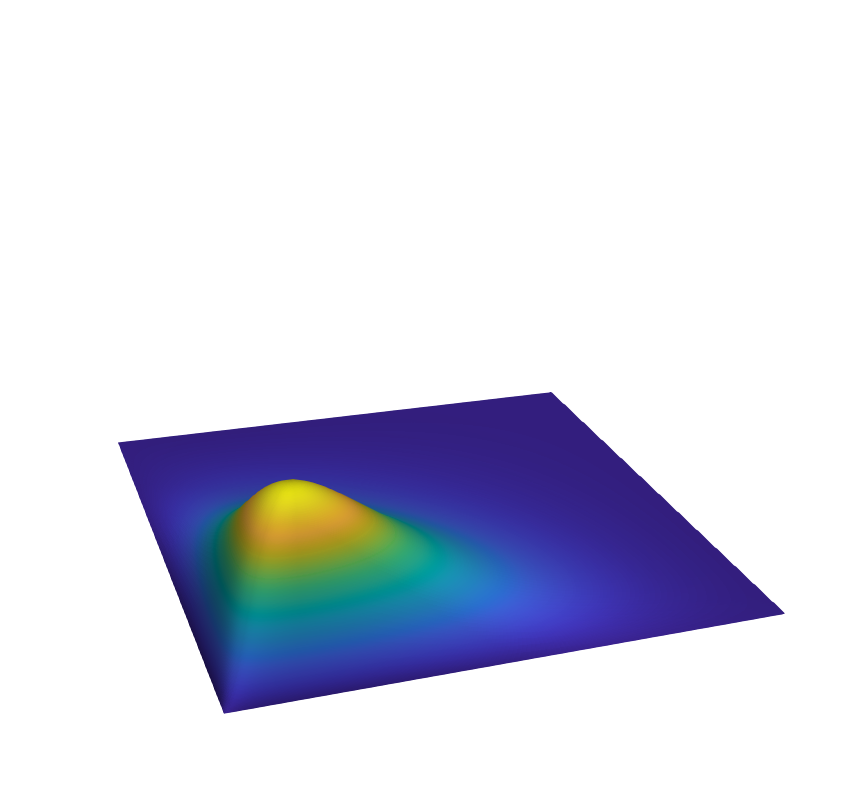}
			\caption{$u(\bx)$ for $\kappa=10$.}
			\label{fig:sol_conv_10}
		\end{subfigure}
	\begin{subfigure}[t]{\subfigsizesm\textwidth}
		\centering
		\begin{tikzpicture}
			\node at (0,0) {\includegraphics[trim=4cm 3cm 2.3cm 6.2cm, clip, width=\textwidth]{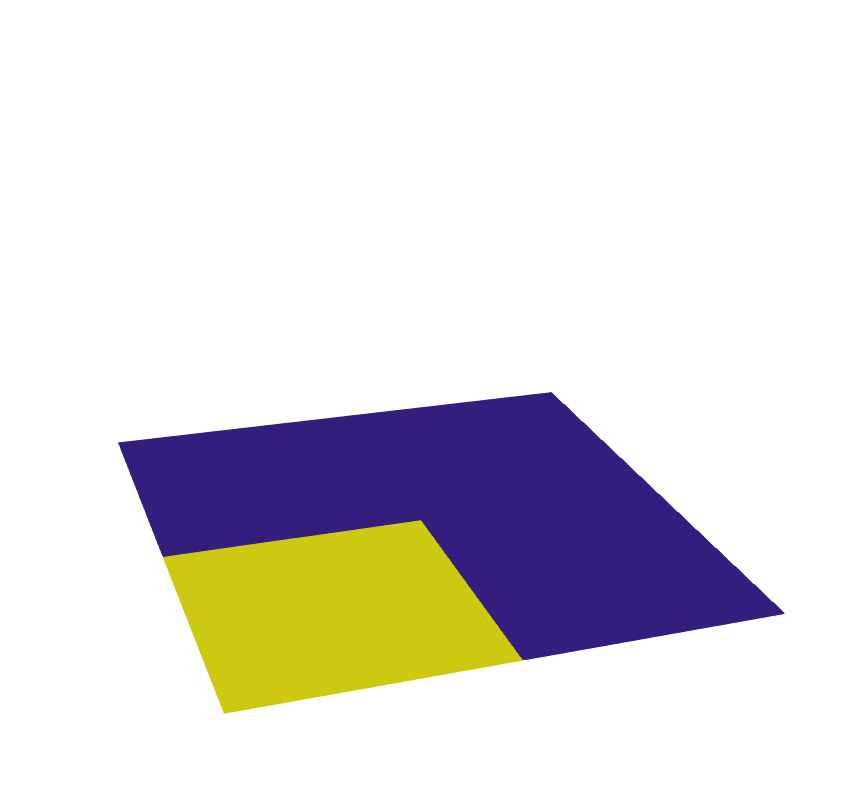}};
		\end{tikzpicture}
		\caption{Domains $\Omega_2\subset\Omega_1$.}
		\label{fig:domains_priori}
	\end{subfigure}
	\end{center}
	\caption{Solution $u(\bx)$ in \cref{eq:solsmooth} for two values of $\kappa$ and local domains $\Omega_1$, $\Omega_2$.}
	\label{fig:sol_dom_conv}
\end{figure}
Let $h$ be the grid size of $\widehat{\mathcal M}_1$, then the grid size of $\widehat{\mathcal{M}}_2$ is $h/2$. 
For different choices of $h$ we run \cref{alg:local} without calling LocalDomain, since the local domains and meshes are chosen beforehand. After the second iteration we compute the exact energy error and the error estimators. The results are reported in \cref{tab:conv_a,tab:conv_b} for $\kappa=100$ and $\kappa=10$, respectively. We recall that $\eta_{NC}$ measures the non conformity of $u_k$, $\eta_{R}$ measures the error in the energy conservation, $\eta_{DF}$ the difference between $-A\nabla u_k$ and the reconstructed diffusive flux $\btk$, $\eta_U,\tilde\eta_{U}$ are upwind errors and $\eta_{\Gamma,1},\eta_{\Gamma,2}$ measure the jumps of $\btk,\bqk$ across subdomains boundaries.

\begin{table}
	\begin{filecontents*}{data/corner/local_sing_1e-2_diff_1e0_a_posteriori_data.csv}
level, erren, erraug, etafull, etatfull, eff, efft, etaNC, etaR, etaDF, etaC1, etaC2, etaU, etatC1, etatU, etaGu, etaGd, lochmax, elements, linsolvernow, linsolvertot, linsystemnow, linsystemtot
2, 4.1767686344488469e-01, 4.1910313913694275e-01, 8.1918568307781137e-01, 2.3242039666200602e+00, 1.9612905448518063e+00, 5.5456610785743177e+00, 1.2925756305644595e-01, 1.1195770795452037e-01, 5.6844629177274419e-01, 0.0000000000000000e+00, 0.0000000000000000e+00, 4.5326672851093743e-03, 0.0000000000000000e+00, 2.7328310339255990e-04, 1.0169280642342816e-03, 1.0276946489315793e-01, 2.2097086912079612e-02, 8192, 6.7889800000000000e-01, 1.1637459999999999e+00, 8.1506300000000009e-01, 1.4533920000000000e+00
2, 2.1183420276212503e-01, 2.1224368061346602e-01, 3.8256705855858419e-01, 1.0893474147474107e+00, 1.8059739813979907e+00, 5.1325316805606507e+00, 5.7594284902635129e-02, 2.9039576371298895e-02, 2.8548716754756293e-01, 0.0000000000000000e+00, 0.0000000000000000e+00, 9.1275211187275549e-04, 0.0000000000000000e+00, 4.6953923919615472e-05, 8.2843060123312493e-04, 7.9852844746862855e-02, 1.1048543456039806e-02, 32768, 7.8879190000000010e+00, 1.3916104000000001e+01, 8.4074230000000014e+00, 1.4943846000000001e+01
2, 1.0581630573539416e-01, 1.0592694774450503e-01, 1.8593684720757914e-01, 5.3213603253949071e-01, 1.7571663073603760e+00, 5.0236133851699254e+00, 2.5536900584353758e-02, 7.3399015907306938e-03, 1.4238695499376938e-01, 0.0000000000000000e+00, 0.0000000000000000e+00, 1.7871865500256311e-04, 0.0000000000000000e+00, 8.6969110948605731e-06, 6.1968972016065017e-04, 5.6726767748088165e-02, 5.5242717280199029e-03, 131072, 1.0509723899999999e+02, 1.9822391099999999e+02, 1.0713846500000001e+02, 2.0224623600000001e+02
2, 5.2748431809222958e-02, 5.2777254422227290e-02, 9.4603265953449950e-02, 2.7214595451408258e-01, 1.7934801606160500e+00, 5.1565007974243455e+00, 1.1638143677584268e-02, 1.8404367088611204e-03, 7.1004091650234219e-02, 0.0000000000000000e+00, 0.0000000000000000e+00, 3.5545028588712140e-05, 0.0000000000000000e+00, 2.1152287282299628e-06, 4.4765954411371111e-04, 3.9205614703662711e-02, 2.7621358640099515e-03, 524288, 1.0433949639999998e+03, 2.3154251329999997e+03, 1.0510837389999999e+03, 2.3310563569999999e+03
\end{filecontents*}
\csvreader[
	before reading=\small\centering\sisetup{table-number-alignment=left,table-parse-only,zero-decimal-to-integer,round-mode=figures,round-precision=2,output-exponent-marker = \text{e},fixed-exponent=0},
	tabular={lSSSSSSSS},head to column names,
	table head=\toprule $h$ & \text{$\nB{u-u_k}$} &$\eta_{NC}$ & $\eta_{R}$ & $\eta_{DF}$ & $\eta_{U}$ & \text{$\tilde{\eta}_{U}$} & $\eta_{\Gamma,1}$ & $\eta_{\Gamma,2}$ \\\midrule,
	late after last line=\\\toprule Order & $1$ & $1$ & $2$ & $1$ & $2$ & $2$ & \text{$0.5$} & \text{$0.5$} \\\midrule]
	{data/corner/local_sing_1e-2_diff_1e0_a_posteriori_data.csv}{}
	{$2^{-{\the\numexpr\thecsvrow+5\relax}}$ & \erren & \etaNC & \etaR & \etaDF & \etaU & \etatU & \etaGu & \etaGd}
	\caption{Convergence rate of error estimators for $\kappa=100$.}
	\label{tab:conv_a}
\end{table}
\begin{table}
	\begin{filecontents*}{data/corner/local_sing_1e-1_diff_1e0_a_posteriori_data.csv}
level, erren, erraug, etafull, etatfull, eff, efft, etaNC, etaR, etaDF, etaC1, etaC2, etaU, etatC1, etatU, etaGu, etaGd, lochmax, elements, linsolvernow, linsolvertot, linsystemnow, linsystemtot
2, 1.8033578330892211e-02, 1.8135891964946610e-02, 1.9886960375181612e-01, 5.9197330454650965e-01, 1.1027739481473006e+01, 3.2640980972465357e+01, 4.6141665365038095e-03, 6.2846196790977726e-04, 2.3233160419422846e-02, 0.0000000000000000e+00, 0.0000000000000000e+00, 3.4550624751449993e-04, 0.0000000000000000e+00, 3.1817017014846780e-04, 1.0115584984549687e-02, 1.8021163996520140e-01, 2.2097086912079612e-02, 8192, 3.6723899999999998e-01, 8.6092199999999997e-01, 4.9123299999999992e-01, 1.1353059999999999e+00
2, 9.0091524671687562e-03, 9.0353405299452676e-03, 1.3674377700110055e-01, 4.0802731609006726e-01, 1.5178317549783245e+01, 4.5159041293216085e+01, 2.2006444743933491e-03, 1.5755099698460352e-04, 1.1603987888466305e-02, 0.0000000000000000e+00, 0.0000000000000000e+00, 1.1717441315721048e-04, 0.0000000000000000e+00, 1.1329679251430763e-04, 7.2283522477725329e-03, 1.2599096880029306e-01, 1.1048543456039806e-02, 32768, 3.2248290000000006e+00, 9.8939350000000008e+00, 3.7104390000000000e+00, 1.0905099000000000e+01
2, 4.5017766143337096e-03, 4.5085389870044993e-03, 9.4983527846775745e-02, 2.8387755239785395e-01, 2.1099120632584718e+01, 6.2964422225494367e+01, 1.0724731102229911e-03, 3.9416025629993472e-05, 5.7977033434678995e-03, 0.0000000000000000e+00, 0.0000000000000000e+00, 4.0773314029984509e-05, 0.0000000000000000e+00, 4.0182181646651160e-05, 5.1354549590431993e-03, 8.8344253068185746e-02, 5.5242717280199029e-03, 131072, 1.4925674000000001e+01, 7.9454999000000001e+01, 1.6827455000000000e+01, 8.3215192999999999e+01
2, 2.2500962625675303e-03, 2.2523085719886879e-03, 6.6455026996319444e-02, 1.9883590854275768e-01, 2.9534303977061509e+01, 8.8280935843171108e+01, 5.2907700365162466e-04, 9.8560985195111122e-06, 2.8976162515782404e-03, 0.0000000000000000e+00, 0.0000000000000000e+00, 1.4322828575665503e-05, 0.0000000000000000e+00, 1.4227108345710041e-05, 3.6393869626689208e-03, 6.2158444605419576e-02, 2.7621358640099515e-03, 524288, 5.3323036999999999e+01, 9.5581267100000002e+02, 6.0869876999999974e+01, 9.7089304500000003e+02
\end{filecontents*}
\csvreader[
	before reading=\small\centering\sisetup{table-number-alignment=left,table-parse-only,zero-decimal-to-integer,round-mode=figures,round-precision=2,output-exponent-marker = \text{e},fixed-exponent=0},
	tabular={lSSSSSSSS},
	head to column names,
	table head=\toprule $h$ & \text{$\nB{u-u_k}$} & $\eta_{NC}$ & $\eta_{R}$ & $\eta_{DF}$ & $\eta_{U}$ & \text{$\tilde{\eta}_{U}$} & $\eta_{\Gamma,1}$ & $\eta_{\Gamma,2}$ \\\midrule,
	late after last line=\\\toprule Order & $1$ & $1$ & $2$ & $1$ & $1.5$ & $1.5$ & \text{$0.5$} & \text{$0.5$} \\\midrule]
	{data/corner/local_sing_1e-1_diff_1e0_a_posteriori_data.csv}{}
	{$2^{-{\the\numexpr\thecsvrow+5\relax}}$ & \erren & \etaNC & \etaR & \etaDF & \etaU & \etatU & \etaGu & \etaGd}
	\caption{Convergence rate of error estimators for $\kappa=10$.}
	\label{tab:conv_b}
\end{table}

We see that the energy error converges with order one, as predicted by the a priori error analysis of \cite{AbR19}. We also observe that the error estimators $\eta_{\Gamma,1}$ and $\eta_{\Gamma,2}$ measuring the reconstructed fluxes' jumps across subdomains' boundaries have a lower rate of convergence. Therefore, the error estimators are not efficient, in the sense that they cannot be bounded from above by the energy error multiplied by a mesh-size independent constant.
However, the relative size of $\eta_{\Gamma,1}$, $\eta_{\Gamma,2}$ compared to the other estimators gives an information on the suitability of the local scheme:
\begin{itemize}
	\item if $\eta_{\Gamma,1}$, $\eta_{\Gamma,2}$ are comparable to the other estimators one should use the local scheme. The typical situation is when the errors are localized, with local regions covering the large error regions (see \cref{fig:sol_conv_100,fig:domains_priori} and \cref{tab:conv_a});
	\item if the relative size of $\eta_{\Gamma,1}$, $\eta_{\Gamma,2}$ is larger than the other estimators, this is an indication that one should switch from local to classical method. The typical situation is when the errors are not (or less) localized (see \cref{fig:sol_conv_10,fig:domains_priori} and \cref{tab:conv_b}). On purpose we did choose a local domain that is too small to cover the error region.
\end{itemize}

In the next experiments we let the scheme select the local subdomains on the fly, using the fixed energy fraction marking strategy \cite[Section 4.2]{Dor96} implemented in the $\text{LocalDomain}(u_k,\mathfrak{T}_k)$ routine of \cref{alg:local}. First, we revisit the example of \cref{exp:conv}. Second, we consider two examples where the errors are localized, illustrating the efficiency of the algorithm.

\subsection{A nonlocal smooth problem}\label{exp:corner}
Considering the same problem as in \cref{exp:conv} with $\kappa=10$, we run the local and classical schemes for $k=1,\ldots,15$ starting with a uniform mesh of 128 elements. Here, we employ the automatic subdomains' identification algorithm and the goal is to show when one should switch from local to nonlocal methods.
As the error is distributed in the whole domain, it is not possible to chose the subdomains $\Omega_{k}$ so that the errors at their boundaries are negligible. Consequently, the error estimators $\eta_{\Gamma,1}$, $\eta_{\Gamma,2}$ will dominate.
Indeed, we see in \cref{tab:dom} that the error estimators $\eta_{\Gamma,1}$, $\eta_{\Gamma,2}$ measuring the reconstructed fluxes' jumps dominate the other estimators.
\begin{table}
	\begin{filecontents*}{data/corner/SPA2FFM_sing_1_diff_0_b_1_nref_3_lay_21_a_posteriori_data_first_5_levels.csv}
level, erren, erraug, etafull, etatfull, eff, efft, etaNC, etaR, etaDF, etaC1, etaC2, etaU, etatC1, etatU, etaGu, etaGd, lochmax, elements, linsolvernow, linsolvertot, linsystemnow, linsystemtot
1, 2.4915120607837926e-01, 2.6084342420635248e-01, 5.4253989200763419e-01, 1.4918877057087623e+00, 2.1775527421567418e+00, 5.7194760046108515e+00, 9.1569699315581968e-02, 9.5839965804019212e-02, 3.1801695596730362e-01, 0.0000000000000000e+00, 0.0000000000000000e+00, 4.7667503630216532e-02, 0.0000000000000000e+00, 3.8446956952622384e-03, 0, 0, 1.7677669529663689e-01, 128, 1.4419999999999999e-03, 1.4419999999999999e-03, 5.8319999999999995e-03, 5.8319999999999995e-03
2, 1.4140221787481805e-01, 1.4651082502712906e-01, 3.1509939577469082e-01, 8.8272553344296967e-01, 2.2283907601339403e+00, 6.0249850704172712e+00, 5.1272689335730613e-02, 4.7732729721548905e-02, 1.8264482847707303e-01, 0.0000000000000000e+00, 0.0000000000000000e+00, 1.3866607037882899e-02, 0.0000000000000000e+00, 1.7210011589888011e-03, 7.2082397768105011e-16, 6.7452708689360374e-02, 1.7677669529663689e-01, 108, 1.3410000000000000e-03, 2.7829999999999999e-03, 3.7980000000000002e-03, 9.6299999999999997e-03
3, 9.1547316071945745e-02, 9.5019962029463237e-02, 3.3138460417882731e-01, 9.6096379150063105e-01, 3.6198177991192755e+00, 1.0113283261496790e+01, 2.9592122705133839e-02, 3.8555123118153449e-02, 1.1952531477558985e-01, 0.0000000000000000e+00, 0.0000000000000000e+00, 5.8698248675568328e-03, 0.0000000000000000e+00, 1.3662466162353854e-03, 7.4513546899162707e-04, 1.8673522960063804e-01, 1.7677669529663689e-01, 234, 2.3039999999999996e-03, 5.0869999999999995e-03, 7.4180000000000010e-03, 1.7048000000000001e-02
4, 8.0516192172023701e-02, 8.3709809549086259e-02, 2.9074324889617192e-01, 8.4427622016534642e-01, 3.6109910448198530e+00, 1.0085750101608756e+01, 2.5552517996026543e-02, 3.0467363581234603e-02, 1.0405418939543119e-01, 0.0000000000000000e+00, 0.0000000000000000e+00, 4.4142081128171872e-03, 0.0000000000000000e+00, 1.2376213541833215e-03, 5.9475496939803956e-04, 1.7891608554932015e-01, 1.7677669529663689e-01, 122, 1.0530000000000001e-03, 6.1399999999999996e-03, 3.5439999999999985e-03, 2.0591999999999999e-02
5, 7.2266964387650789e-02, 7.4934528510902695e-02, 4.2358027490046674e-01, 1.2444526993887524e+00, 5.8613265202107963e+00, 1.6607199966670759e+01, 2.4919074748204300e-02, 2.5030581380428605e-02, 9.2853872113802860e-02, 0.0000000000000000e+00, 0.0000000000000000e+00, 3.9815910222706112e-03, 0.0000000000000000e+00, 1.5567904975389989e-03, 1.0676674935694033e-02, 3.4241227490728865e-01, 1.7677669529663689e-01, 352, 2.6669999999999992e-03, 8.8069999999999989e-03, 9.2049999999999979e-03, 2.9796999999999997e-02
\end{filecontents*}
\csvreader[
	before reading=\small\centering\sisetup{table-number-alignment=left,table-parse-only,zero-decimal-to-integer,round-mode=figures,round-precision=2,output-exponent-marker = \text{e},fixed-exponent=0},
	tabular={lSSSSSSSS},
	head to column names,
	table head=\toprule $k$ & \text{$\nB{u-u_k}$} & $\eta_{NC}$ & $\eta_{R}$ & $\eta_{DF}$ & $\eta_{U}$ & \text{$\tilde{\eta}_{U}$} & $\eta_{\Gamma,1}$ & $\eta_{\Gamma,2}$ \\\midrule,
	]
	{data/corner/SPA2FFM_sing_1_diff_0_b_1_nref_3_lay_21_a_posteriori_data_first_5_levels.csv}{}
	{\level & \erren & \etaNC & \etaR & \etaDF & \etaU & \etatU & \etaGu & \etaGd}
	\caption{\Cref{exp:corner}, nonlocal smooth problem. Dominance of $\eta_{\Gamma,1}$ and $\eta_{\Gamma,2}$ over the other error estimators. Only the results of the first five iterations are shown, i.e. $k\leq 5$.}
	\label{tab:dom}
\end{table}
This phenomenon brings two issues into the algorithm. First, the effectivity index of the local scheme is significantly larger than the index for the classical scheme, as we illustrate in \cref{fig:corner_effind_eta}. Second, the marking error estimator $\eta_{M,K}$ \cref{eq:marketa} will be larger at the boundaries of the local domains than in the large error regions; indeed, we see in \cref{fig:corner_doms} that the local domain $\Omega_4$ chosen by the algorithm do not correspond to a large error region but is in a neighborhood of the boundary of $\Omega_3$, where $\eta_{\Gamma,1}$, $\eta_{\Gamma,2}$ are large. For this reason the algorithm in unable to detect the high error regions and we see in \cref{fig:corner_effenerr}, where we show the computational cost in function of the energy errors, that the error of the local method stagnates.
\begin{figure}
	\begin{center}
		\begin{subfigure}[t]{\subfigsize\textwidth}
			\centering
			\begin{tikzpicture}[scale=\plotimscale]
				\begin{semilogyaxis}[height=\aspectratio*\plotimsized\textwidth,width=\plotimsized\textwidth,legend style={at={(0,1)},anchor=north west},xlabel={Iteration $k$}, ylabel={Effectivity index of $\eta$},label style={font=\normalsize},tick label style={font=\normalsize},legend image post style={scale=\legendmarkscale},legend style={nodes={scale=\legendfontscale, transform shape},draw=none}, log basis y=10,ymin=1,ymax=400]
					\addplot+[color=OrangeRed,mark=o,line width=\plotlinewidth pt,mark size=\plotmarksizeu pt] table [x=level,y=eff,col sep=comma] 
					{data/corner/SPA2FFM_sing_1_diff_0_b_1_nref_3_lay_21_a_posteriori_data.csv};\addlegendentry{Local}
					\addplot+[color=ForestGreen,mark=star,line width=\plotlinewidth pt,mark size=\plotmarksizeu pt] table[x=level,y=eff,col sep=comma] 
					{data/corner/SPA1_sing_1_diff_0_b_1_nref_3_lay_21_a_posteriori_data.csv};\addlegendentry{Classical}
				\end{semilogyaxis}
			\end{tikzpicture}
			\caption{Effectivity index of $\eta$.}
			\label{fig:corner_effind_eta}
		\end{subfigure}\hfill
		\begin{subfigure}[t]{\subfigsize\textwidth}
			\centering
			\begin{tikzpicture}[scale=\plotimscale]
				\begin{loglogaxis}[height=\aspectratio*\plotimsized\textwidth,width=\plotimsized\textwidth, x dir=reverse,legend style={at={(0,1)},anchor=north west},
					xlabel={Energy norm error.}, ylabel={GMRES cost [sec.]},log basis x={2},label style={font=\normalsize},tick label style={font=\normalsize},legend image post style={scale=\legendmarkscale},legend style={nodes={scale=\legendfontscale, transform shape},draw=none}]
					\addplot+[color=OrangeRed,mark=o,line width=\plotlinewidth pt,mark size=\plotmarksizeu pt] table [x=erren,y=linsolvertot,col sep=comma,select coords between index={0}{14}] 
					{data/corner/SPA2FFM_sing_1_diff_0_b_1_nref_3_lay_21_a_posteriori_data.csv};\addlegendentry{Local}
					\addplot+[color=ForestGreen,mark=star,line width=\plotlinewidth pt,mark size=\plotmarksizeu pt] table[x=erren,y=linsolvertot,col sep=comma,select coords between index={0}{14}] 
					{data/corner/SPA1_sing_1_diff_0_b_1_nref_3_lay_21_a_posteriori_data.csv};\addlegendentry{Classical}
				\end{loglogaxis}
			\end{tikzpicture}
			\caption{GMRES cost versus energy norm error.}
			\label{fig:corner_effenerr}
		\end{subfigure}
	\end{center}
	\caption{\Cref{exp:corner}, nonlocal smooth problem. Effectivity indexes in function of the iteration number.}
\end{figure}
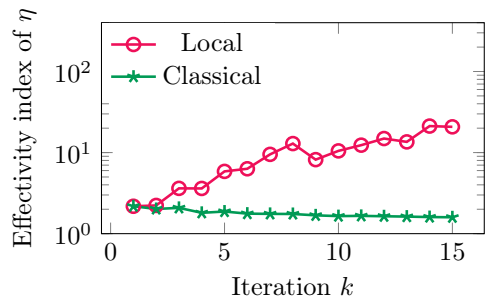
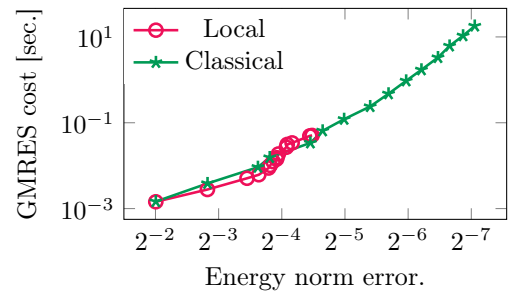
\begin{figure}
	\begin{center}
		\begin{tikzpicture}
			\node at (0,0) {\includegraphics[trim=4cm 3cm 2.3cm 6.2cm, clip, width=0.22\textheight]{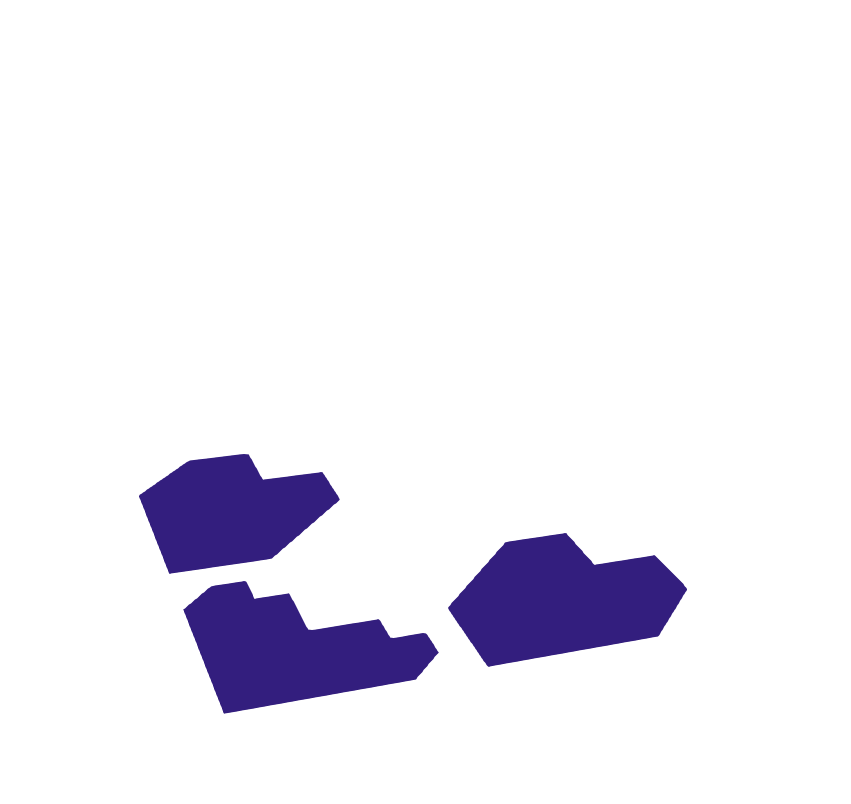}};
			\node[opacity=0.7] at (0,0) {\includegraphics[trim=4cm 3cm 2.3cm 6.2cm, clip, width=0.22\textheight]{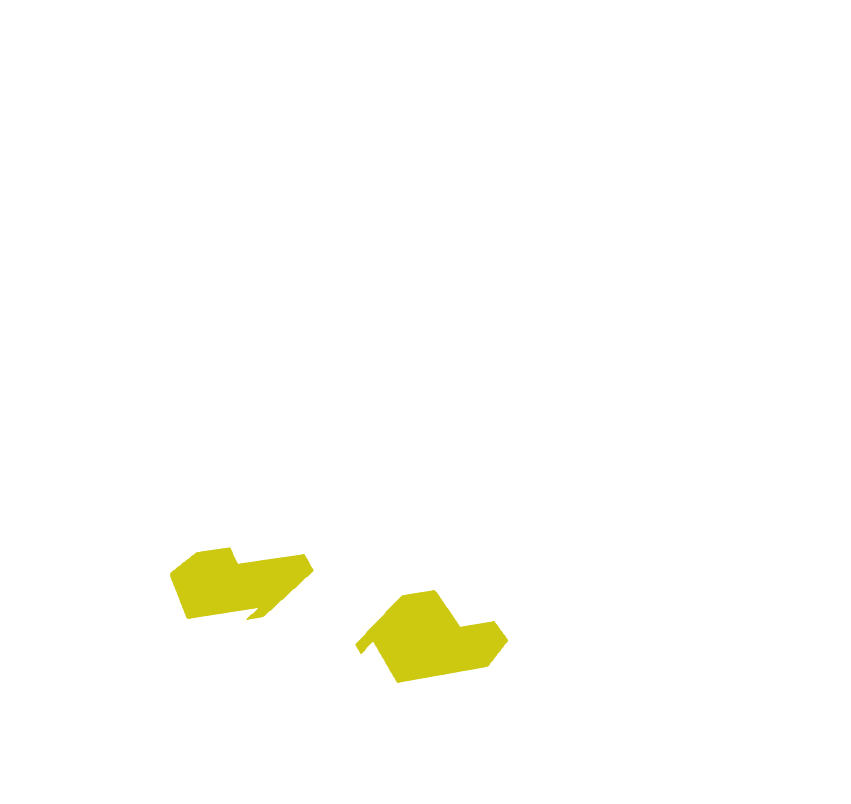}};
		\end{tikzpicture}
		\caption{Local domains $\Omega_3$ (darker) and $\Omega_4$ (brighter).}
		\label{fig:corner_doms}
	\end{center}
\end{figure}

This example shows that if the errors are not localized then the estimators $\eta_{\Gamma,1}$, $\eta_{\Gamma,2}$ dominate, the local scheme becomes inefficient and a classical \emph{global} method should be preferred over a local method. However, our algorithm allows to monitor the size of the error estimators $\eta_{\Gamma,1}$ and $\eta_{\Gamma,2}$ and when these error estimators start to dominate the other error indicators (as seen in \cref{tab:dom}) it provides a switching criteria.

\subsection{Reaction dominated problem}\label{exp:bndlayer_sym}
In our next example we consider a symmetric problem and want to compare the local and classical schemes (\cref{alg:local,alg:classical}) in a singularly perturbed regime. We investigate the efficiency measured as the computational cost and analyze their effectivity indexes. The setting is as follows: we solve \eqref{eq:elliptic} in $\Omega=[0,1]\times [0,1]$ with $\varepsilon=10^{-6}$, $A=\varepsilon I_2$, $\bbeta=(0,0)^\top$, $\mu=1$ and we choose $f$ such that the exact solution is given by
\begin{equation}\label{eq:bndlayer}
u(\bx)=e^{x_1+x_2}\left( x_1-\frac{1-e^{-\zeta x_1}}{1-e^{-\zeta}}\right)\left(x_2-\frac{1-e^{-\zeta x_2}}{1-e^{-\zeta}} \right),
\end{equation}
where $\zeta=10^{4}$. The solution is illustrated in \cref{fig:bndlayer_sol}. 

\begin{figure}
	\begin{center}
		\begin{subfigure}[t]{\subfigsize\textwidth}
			\centering
			\includegraphics[trim=0cm 0cm 0cm 0cm, clip, width=0.22\textheight]{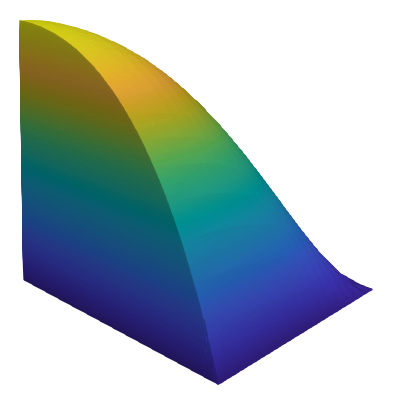}
			\caption{Solution $u(\bx)$.}
			\label{fig:bndlayer_sol}
		\end{subfigure}
		\begin{subfigure}[t]{\subfigsize\textwidth}
			\centering
			\begin{tikzpicture}[spy using outlines= {circle, connect spies,every spy on node/.append style={thick}}]
				\coordinate (spypoint) at (-0.1,0.15);
				\coordinate (magnifyglass) at (1.5,0.5);
				\coordinate (spypoint_bis) at (-0.03,-0.55);
				\coordinate (magnifyglass_bis) at (1.5,-1.2);
				\node at (0,0) {\includegraphics[trim=0cm 0cm 0cm 0cm, clip, width=0.22\textheight]{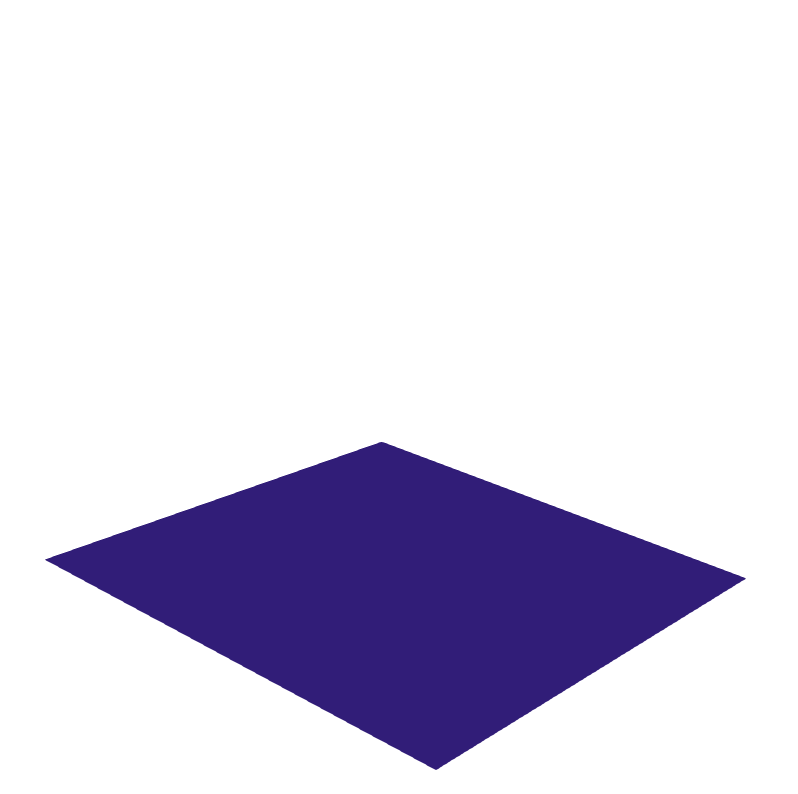}};
				\node at (0,0) {\includegraphics[trim=0cm 0cm 0cm 0cm, clip, width=0.22\textheight]{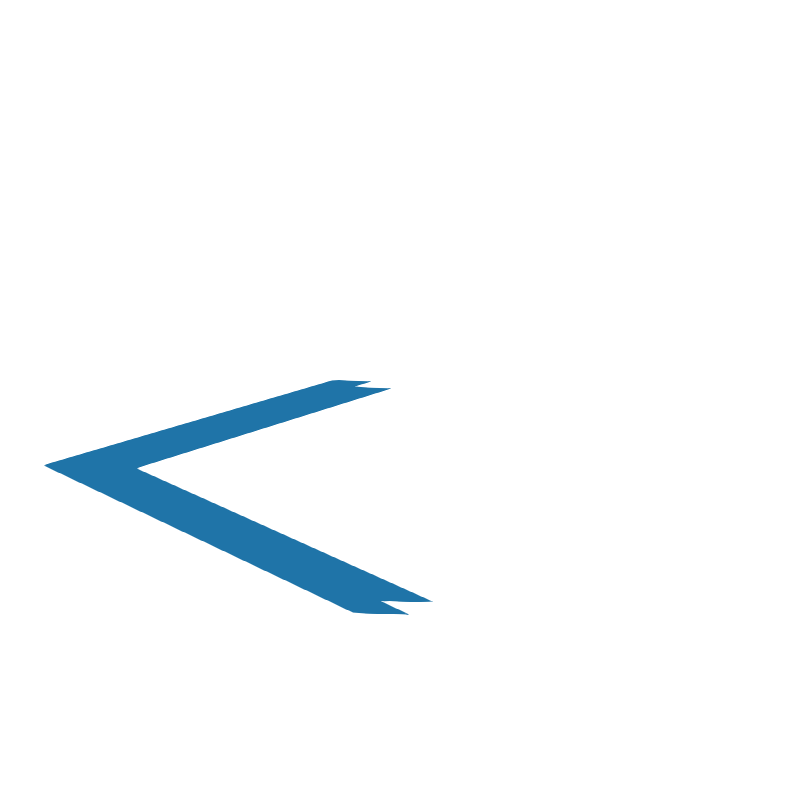}};
				\node at (0,0) {\includegraphics[trim=0cm 0cm 0cm 0cm, clip, width=0.22\textheight]{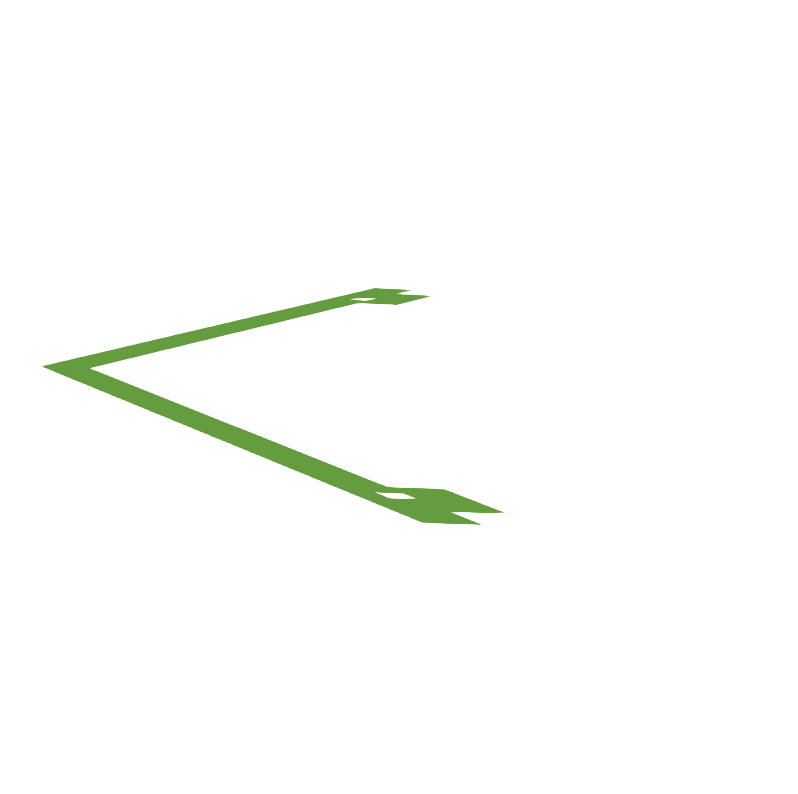}};
				\node at (0,0) {\includegraphics[trim=0cm 0cm 0cm 0cm, clip, width=0.22\textheight]{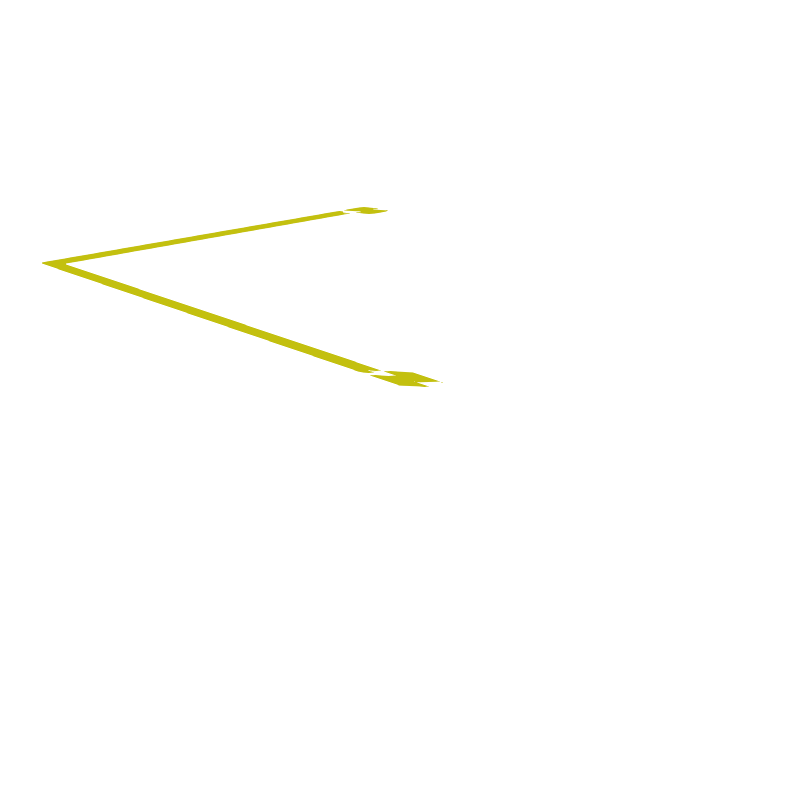}};
				\spy [WildStrawberry, size=1.3cm, magnification=4] on (spypoint) in node[fill=white] at (magnifyglass);
				\spy [WildStrawberry, size=1.3cm, magnification=4] on (spypoint_bis) in node[fill=white] at (magnifyglass_bis);
			\end{tikzpicture}
			\caption{First local domains $\Omega_k$, $k=1,\ldots,4$.}
			\label{fig:bndlayer_doms}
		\end{subfigure}
	\end{center}
	\caption{Solution $u(\bx)$ in \eqref{eq:bndlayer} of the reaction dominated problem and first local domains chosen by the error estimators.}
\end{figure}
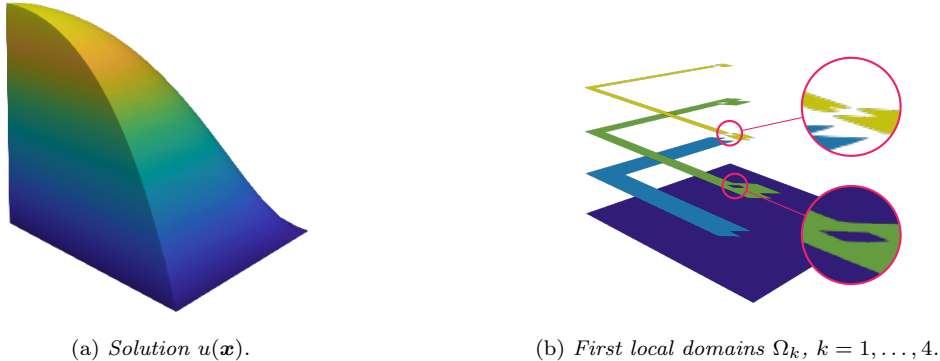

Since the problem is symmetric we have $\nB{{\cdot}}=\nBp{{\cdot}}$, but their related error estimators $\eta$ and $\tilde\eta$, respectively, satisfy $\tilde\eta>\eta$ and hence the effectivity index of $\eta$ will be lower (see  \cref{thm:energynormbound,thm:augmentednormbound}). 

Starting from a coarse mesh (128 elements), we let the two algorithms run for $k=1,\ldots,20$. In \cref{fig:bndlayer_doms} we show the first four subdomains $\Omega_k$ chosen by the local scheme. Note that the local domain $\Omega_4$ chosen by the algorithm is disconnected, while subdomain $\Omega_3$ has an hole; as is allowed by the theory. Several of the subsequent subdomains (not displayed) are also disconnected or contain holes. The first iterations are needed to capture the boundary layer and reach the convergence regime, hence we will plot the results for $k\geq 7$. The most expensive part of the code is the solution of linear systems by means of the conjugate gradient (CG) method preconditioned with the incomplete Cholesky factorization, followed by the computation of the potential and fluxes reconstruction and then by the evaluation of the error estimators. In the local scheme, the time spent doing these tasks is proportional to the number of elements inside each subdomain $\Ok$. For the classical scheme, the cost of these tasks depends on the total number of elements in the mesh. Since the CG routine is the most expensive part, we take the time spent in it as an indicator for the computational cost.

In \cref{fig:bndlayer_sym_etacost}, we plot the simulation cost against the error estimator $\eta$, for both the local and classical algorithms. Each circle or star in the figure represents an iteration $k$. We observe that the local scheme provides similar error bounds but at a smaller cost. The effectivity index of $\eta$ at each iteration $k$ is shown in \cref{fig:bndlayer_sym_etaeffind}, we can observe that the local scheme has an effectivity index similar to the classical scheme.
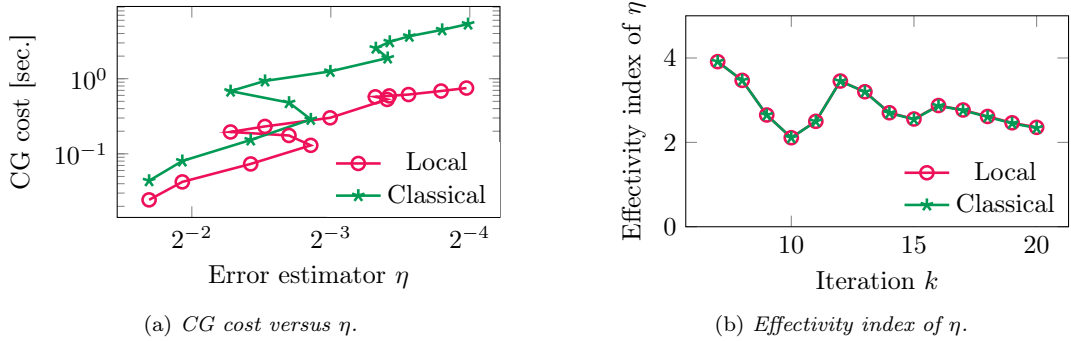
\begin{figure}
	\begin{center}
		\begin{subfigure}[t]{\subfigsize\textwidth}
			\centering
			\begin{tikzpicture}[scale=\plotimscale]
			\begin{loglogaxis}[height=\aspectratio*\plotimsized\textwidth,width=\plotimsized\textwidth, x dir=reverse,legend style={at={(1,0)},anchor=south east},
			xlabel={Error estimator $\eta$}, ylabel={CG cost [sec.]},log basis x={2},label style={font=\normalsize},tick label style={font=\normalsize},legend image post style={scale=\legendmarkscale},legend style={nodes={scale=\legendfontscale, transform shape},draw=none}]
			\addplot[color=OrangeRed,mark=o,line width=\plotlinewidth pt,mark size=\plotmarksizeu pt] table [x=etafull,y=linsolvertot,col sep=comma,select coords between index={6}{19}] 
			{data/bndlayer/SPA2FFM_sing_4_diff_6_b_0_nref_3_lay_21_a_posteriori_data.csv};\addlegendentry{Local}
			\addplot[color=ForestGreen,mark=star,line width=\plotlinewidth pt,mark size=\plotmarksizeu pt] table[x=etafull,y=linsolvertot,col sep=comma,select coords between index={6}{19}] 
			{data/bndlayer/SPA1_sing_4_diff_6_b_0_nref_3_lay_21_a_posteriori_data.csv};\addlegendentry{Classical}
			\end{loglogaxis}
			\end{tikzpicture}
			\caption{CG cost versus $\eta$.}
			\label{fig:bndlayer_sym_etacost}
		\end{subfigure}\hfill
		\begin{subfigure}[t]{\subfigsize\textwidth}
			\centering
			\begin{tikzpicture}[scale=\plotimscale]
			\begin{axis}[height=\aspectratio*\plotimsized\textwidth,width=\plotimsized\textwidth,legend style={at={(1,0)},anchor=south east},
			xlabel={Iteration $k$}, ylabel={Effectivity index of $\eta$},ymin=0,ymax=5,label style={font=\normalsize},tick label style={font=\normalsize},legend image post style={scale=\legendmarkscale},legend style={nodes={scale=\legendfontscale, transform shape},draw=none}]
			\addplot[color=OrangeRed,mark=o,line width=\plotlinewidth pt,mark size=\plotmarksizeu pt] table [x=level,y=eff,col sep=comma,select coords between index={6}{19}] 
			{data/bndlayer/SPA2FFM_sing_4_diff_6_b_0_nref_3_lay_21_a_posteriori_data.csv};\addlegendentry{Local}
			\addplot[color=ForestGreen,mark=star,line width=\plotlinewidth pt,mark size=\plotmarksizeu pt] table[x=level,y=eff,col sep=comma,select coords between index={6}{19}] 
			{data/bndlayer/SPA1_sing_4_diff_6_b_0_nref_3_lay_21_a_posteriori_data.csv};\addlegendentry{Classical}
			\end{axis}
			\end{tikzpicture}
			\caption{Effectivity index of $\eta$.}
			\label{fig:bndlayer_sym_etaeffind}
		\end{subfigure}
	\end{center}
	\caption{\Cref{exp:bndlayer_sym}, reaction dominated problem. Computational cost vs. $\eta$ and effectivity index in function of the iteration number.}
	\label{fig:bndlayer_sym_etacost_etaeffind}
\end{figure}

In \cref{fig:bndlayer_sym_effenerr} we exhibit the cost against the exact energy error and we notice that for some values of $k$ the mesh is refined but the error stays almost constant. This phenomenon significantly increases the simulation cost of the classical scheme without improving the solution. In contrast, the cost of the local scheme increases only marginally. Dividing the two curves in \cref{fig:bndlayer_sym_effenerr} we obtain the relative speed-up, which is plotted in \cref{fig:bndlayer_sym_speedup}. We note that as the error decreases the local scheme becomes faster than the classical scheme.
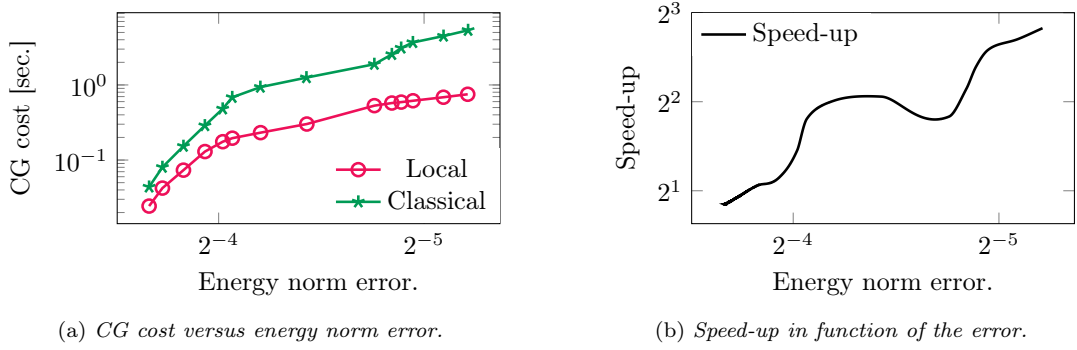
\begin{figure}
	\begin{center}
		\begin{subfigure}[t]{\subfigsize\textwidth}
			\centering
			\begin{tikzpicture}[scale=\plotimscale]
			\begin{loglogaxis}[height=\aspectratio*\plotimsized\textwidth,width=\plotimsized\textwidth, x dir=reverse,legend style={at={(1,0)},anchor=south east},
			xlabel={Energy norm error.}, ylabel={CG cost [sec.]},log basis x={2},label style={font=\normalsize},tick label style={font=\normalsize},legend image post style={scale=\legendmarkscale},legend style={nodes={scale=\legendfontscale, transform shape},draw=none}]
			\addplot[color=OrangeRed,mark=o,line width=\plotlinewidth pt,mark size=\plotmarksizeu pt] table [x=erren,y=linsolvertot,col sep=comma,select coords between index={6}{19}] 
			{data/bndlayer/SPA2FFM_sing_4_diff_6_b_0_nref_3_lay_21_a_posteriori_data.csv};\addlegendentry{Local}
			\addplot[color=ForestGreen,mark=star,line width=\plotlinewidth pt,mark size=\plotmarksizeu pt] table[x=erren,y=linsolvertot,col sep=comma,select coords between index={6}{19}] 
			{data/bndlayer/SPA1_sing_4_diff_6_b_0_nref_3_lay_21_a_posteriori_data.csv};\addlegendentry{Classical}
			\end{loglogaxis}
			\end{tikzpicture}
			\caption{CG cost versus energy norm error.}
			\label{fig:bndlayer_sym_effenerr}
		\end{subfigure}\hfill
		\begin{subfigure}[t]{\subfigsize\textwidth}
			\centering
			\begin{tikzpicture}[scale=\plotimscale]
			\begin{loglogaxis}[height=\aspectratio*\plotimsized\textwidth,width=\plotimsized\textwidth,legend style={at={(0,1)},anchor=north west},
			xlabel={Energy norm error.}, ylabel={Speed-up}, x dir=reverse,log basis x={2},log basis y={2},ymax=8, ytick={2,4,8,16},label style={font=\normalsize},tick label style={font=\normalsize},legend image post style={scale=\legendmarkscale},legend style={nodes={scale=\legendfontscale, transform shape},draw=none}]
			\addplot[color=black,line width=\plotlinewidth pt,mark=none] table[x=erren,y=speeden,col sep=comma] 
			{data/bndlayer/speedup_sing_4_diff_6_b_0_nref_3_lay_21.csv};\addlegendentry{Speed-up}
			\end{loglogaxis}
			\end{tikzpicture}
			\caption{Speed-up in function of the error.}
			\label{fig:bndlayer_sym_speedup}
		\end{subfigure}
	\end{center}
	\caption{\Cref{exp:bndlayer_sym}, reaction dominated problem. Computational cost vs. energy norm error and speed-up in function of the error.}
	\label{fig:bndlayer_sym_effenerr_speedup}
\end{figure}
In \cref{fig:bndlayer_sym_etateffind} we plot the effectivity index of $\tilde\eta$. As expected, for this symmetric problem, it is worse than the effectivity of $\eta$. Finally, we run the same experiment but for different diffusion coefficients $\varepsilon=10^{-4},10^{-6},10^{-8}$ and display in \cref{fig:bndlayer_sym_eta_diff_eps} the effectivity index of $\eta$. We note that it always remains below 4.
\begin{figure}
	\begin{center}
		\begin{subfigure}[t]{\subfigsize\textwidth}
			\centering
			\begin{tikzpicture}[scale=\plotimscale]
			\begin{axis}[height=\aspectratio*\plotimsized\textwidth,width=\plotimsized\textwidth,legend style={at={(1,0)},anchor=south east},
			xlabel={Iteration $k$}, ylabel={Effectivity index of $\tilde\eta$},ymin=0,ymax=15,,label style={font=\normalsize},tick label style={font=\normalsize},legend image post style={scale=\legendmarkscale},legend style={nodes={scale=\legendfontscale, transform shape},draw=none}]
			\addplot[color=OrangeRed,mark=o,line width=\plotlinewidth pt,mark size=\plotmarksizeu pt] table [x=level,y=efft,col sep=comma,select coords between index={6}{19}] 
			{data/bndlayer/SPA2FFM_sing_4_diff_6_b_0_nref_3_lay_21_a_posteriori_data.csv};\addlegendentry{Local}
			\addplot[color=ForestGreen,mark=star,line width=\plotlinewidth pt,mark size=\plotmarksizeu pt] table[x=level,y=efft,col sep=comma,select coords between index={6}{19}] 
			{data/bndlayer/SPA1_sing_4_diff_6_b_0_nref_3_lay_21_a_posteriori_data.csv};\addlegendentry{Classical}
			\end{axis}
			\end{tikzpicture}
			\caption{Effectivity index of $\tilde\eta$.}
			\label{fig:bndlayer_sym_etateffind}
		\end{subfigure}\hfill
		\begin{subfigure}[t]{\subfigsize\textwidth}
			\centering
			\begin{tikzpicture}[scale=\plotimscale]
			\begin{axis}[height=\aspectratio*\plotimsized\textwidth,width=\plotimsized\textwidth,legend columns=2, ,legend style={at={(0,0)},anchor=south west,draw=none,fill=none},
			xlabel={Iteration $k$}, ylabel={Effectivity index of $\eta$},ymin=0,ymax=4.0,label style={font=\normalsize},tick label style={font=\normalsize},legend image post style={scale=\legendmarkscale},legend style={nodes={scale=\legendfontscale, transform shape},draw=none}]
			\addplot[color=ForestGreen,mark=star,line width=\plotlinewidth pt,mark size=\plotmarksizeu pt] table[x=level,y=eff,col sep=comma,select coords between index={6}{19}] 
			{data/bndlayer/SPA2FFM_sing_4_diff_4_b_0_nref_3_lay_21_a_posteriori_data.csv};\addlegendentry{$\varepsilon=10^{-4}$}
			\addlegendimage{empty legend}\addlegendentry{}
			\addplot[color=OrangeRed,mark=o,line width=\plotlinewidth pt,mark size=\plotmarksizeu pt] table [x=level,y=eff,col sep=comma,select coords between index={6}{19}] 
			{data/bndlayer/SPA2FFM_sing_4_diff_6_b_0_nref_3_lay_21_a_posteriori_data.csv};\addlegendentry{$\varepsilon=10^{-6}$}
			\addplot[color=NavyBlue,mark=triangle,line width=\plotlinewidth pt,mark size=\plotmarksizeu pt] table[x=level,y=eff,col sep=comma,select coords between index={6}{19}] 
			{data/bndlayer/SPA2FFM_sing_4_diff_8_b_0_nref_3_lay_21_a_posteriori_data.csv};\addlegendentry{$\varepsilon=10^{-8}$}
			\end{axis}
			\end{tikzpicture}
		\caption{Effectivity index of $\eta$ for different diffusion coefficients $\varepsilon$.}
		\label{fig:bndlayer_sym_eta_diff_eps}
		\end{subfigure}
	\end{center}
	\caption{\Cref{exp:bndlayer_sym}, reaction dominated problem. Effectivity index of $\tilde\eta$ and of $\eta$ but for different diffusion coefficients $\varepsilon$.}
	\label{fig:bndlayer_sym_effetat_effeta}
\end{figure}
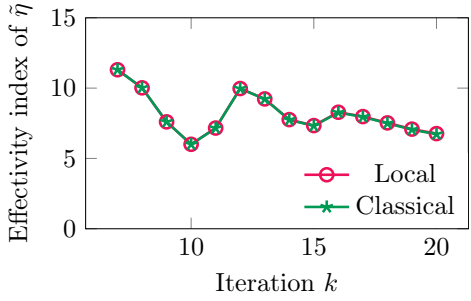
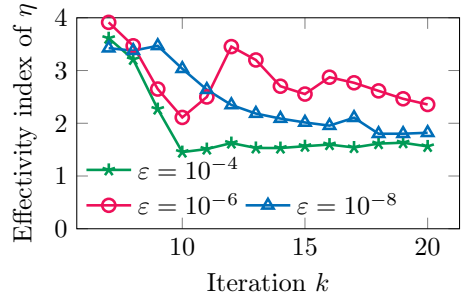

\subsection{Convection dominated problem}\label{exp:bndlayer_notsym}
In this section we perform the same experiment as in \cref{exp:bndlayer_sym} but instead of choosing $\bbeta=(0,0)^\top$ we set $\bbeta=-(1,1)^\top$, hence we solve a nonsymmetric singularly perturbed problem. The linear systems are solved with the GMRES method preconditioned with the incomplete LU factorization. As in \cref{exp:bndlayer_sym}, we investigate the effectivity indexes and efficiency of the local and classical schemes.

For convection dominated problems, the norm $\nBp{{\cdot}}$ is more appropriate than $\nB{{\cdot}}$ since it measures also the error in the advective direction. In \cref{fig:bndlayer_notsym_etatcost}, we plot the simulation cost versus the error estimator $\tilde\eta$, we remark that again the local scheme provides similar error bounds at smaller cost. The effectivity index of $\tilde\eta$ is displayed in \cref{fig:bndlayer_notsym_etateffind}, we note that the local and classical schemes have again similar effectivity indexes.
\begin{figure}
	\begin{center}
		\begin{subfigure}[t]{\subfigsize\textwidth}
			\centering
			\begin{tikzpicture}[scale=\plotimscale]
			\begin{loglogaxis}[height=\aspectratio*\plotimsized\textwidth,width=\plotimsized\textwidth, x dir=reverse,legend style={at={(1,0)},anchor=south east},
			xlabel={Error estimator $\tilde\eta$}, ylabel={GMRES cost [sec.]},log basis x={2},label style={font=\normalsize},tick label style={font=\normalsize},legend image post style={scale=\legendmarkscale},legend style={nodes={scale=\legendfontscale, transform shape},draw=none}]
			\addplot+[color=OrangeRed,mark=o,line width=\plotlinewidth pt,mark size=\plotmarksizeu pt] table [x=etatfull,y=linsolvertot,col sep=comma,select coords between index={6}{19}] 
			{data/bndlayer/SPA2FFM_sing_4_diff_6_b_1_nref_3_lay_21_a_posteriori_data.csv};\addlegendentry{Local}
			\addplot+[color=ForestGreen,mark=star,line width=\plotlinewidth pt,mark size=\plotmarksizeu pt] table[x=etatfull,y=linsolvertot,col sep=comma,select coords between index={6}{19}] 
			{data/bndlayer/SPA1_sing_4_diff_6_b_1_nref_3_lay_21_a_posteriori_data.csv};\addlegendentry{Classical}
			\end{loglogaxis}
			\end{tikzpicture}
			\caption{GMRES cost versus $\tilde\eta$.}
			\label{fig:bndlayer_notsym_etatcost}
		\end{subfigure}
		\begin{subfigure}[t]{\subfigsize\textwidth}
			\centering
			\begin{tikzpicture}[scale=\plotimscale]
			\begin{axis}[height=\aspectratio*\plotimsized\textwidth,width=\plotimsized\textwidth,legend style={at={(1,1)},anchor=north east},
			xlabel={Iteration $k$}, ylabel={Effectivity index of $\tilde\eta$},ymin=0,ymax=15,label style={font=\normalsize},tick label style={font=\normalsize},legend image post style={scale=\legendmarkscale},legend style={nodes={scale=\legendfontscale, transform shape},draw=none}]
			\addplot+[color=OrangeRed,mark=o,line width=\plotlinewidth pt,mark size=\plotmarksizeu pt] table [x=level,y=efft,col sep=comma,select coords between index={6}{19}] 
			{data/bndlayer/SPA2FFM_sing_4_diff_6_b_1_nref_3_lay_21_a_posteriori_data.csv};\addlegendentry{Local}
			\addplot+[color=ForestGreen,mark=star,line width=\plotlinewidth pt,mark size=\plotmarksizeu pt] table[x=level,y=efft,col sep=comma,select coords between index={6}{19}] 
			{data/bndlayer/SPA1_sing_4_diff_6_b_1_nref_3_lay_21_a_posteriori_data.csv};\addlegendentry{Classical}
			\end{axis}
			\end{tikzpicture}
			\caption{Effectivity index of $\tilde\eta$.}
			\label{fig:bndlayer_notsym_etateffind}
		\end{subfigure}
	\end{center}
	\caption{\Cref{exp:bndlayer_notsym}, convection dominated problem. Computational cost vs. $\tilde\eta$ and effectivity index in function of the iteration number.}
	\label{fig:bndlayer_notsym_etatcost_etateffind}
\end{figure}

In \cref{fig:bndlayer_notsym_effenerr_speedup} we plot the simulation cost versus the error in the augmented norm $\nBp{{\cdot}}$ and the relative speed-up. We again observe that the local scheme is faster. 
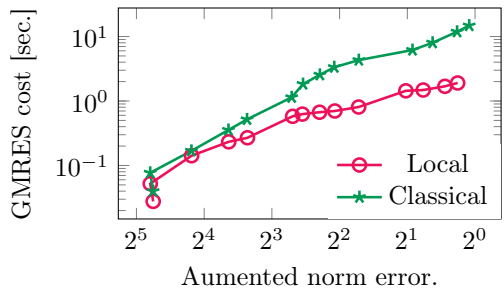
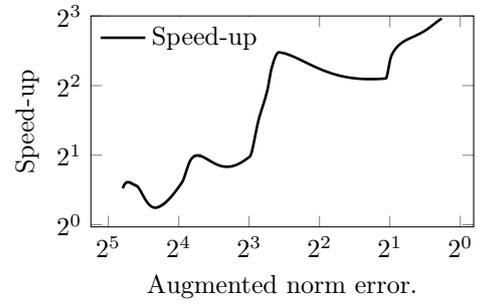
\begin{figure}
	\begin{center}
		\begin{subfigure}[t]{\subfigsize\textwidth}
			\centering
			\begin{tikzpicture}[scale=\plotimscale]
			\begin{loglogaxis}[height=\aspectratio*\plotimsized\textwidth,width=\plotimsized\textwidth, x dir=reverse,legend style={at={(1,0)},anchor=south east},
			xlabel={Aumented norm error.}, ylabel={GMRES cost [sec.]},log basis x={2},label style={font=\normalsize},tick label style={font=\normalsize},legend image post style={scale=\legendmarkscale},legend style={nodes={scale=\legendfontscale, transform shape},draw=none}]
			\addplot+[color=OrangeRed,mark=o,line width=\plotlinewidth pt,mark size=\plotmarksizeu pt] table [x=erraug,y=linsolvertot,col sep=comma,select coords between index={6}{19}] 
			{data/bndlayer/SPA2FFM_sing_4_diff_6_b_1_nref_3_lay_21_a_posteriori_data.csv};\addlegendentry{Local}
			\addplot+[color=ForestGreen,mark=star,line width=\plotlinewidth pt,mark size=\plotmarksizeu pt] table[x=erraug,y=linsolvertot,col sep=comma,select coords between index={6}{19}] 
			{data/bndlayer/SPA1_sing_4_diff_6_b_1_nref_3_lay_21_a_posteriori_data.csv};\addlegendentry{Classical}
			\end{loglogaxis}
			\end{tikzpicture}
			\caption{GMRES cost versus augmented norm error.}
			\label{fig:bndlayer_notsym_effenerr}
		\end{subfigure}
		\begin{subfigure}[t]{\subfigsize\textwidth}
			\centering
			\begin{tikzpicture}[scale=\plotimscale]
			\begin{loglogaxis}[height=\aspectratio*\plotimsized\textwidth,width=\plotimsized\textwidth,legend style={at={(0,1)},anchor=north west},
			xlabel={Augmented norm error.}, ylabel={Speed-up}, x dir=reverse,log basis x={2},log basis y={2},ymax=8,label style={font=\normalsize},tick label style={font=\normalsize},legend image post style={scale=\legendmarkscale},legend style={nodes={scale=\legendfontscale, transform shape},draw=none}]
			\addplot+[color=black,line width=\plotlinewidth pt,mark=none] table[x=erraug,y=speeden,col sep=comma] 
			{data/bndlayer/speedup_sing_4_diff_6_b_1_nref_3_lay_21.csv};\addlegendentry{Speed-up}
			\end{loglogaxis}
			\end{tikzpicture}
			\caption{Speed-up in function of the error.}
			\label{fig:bndlayer_notsym_speedup}
		\end{subfigure}
	\end{center}
	\caption{\Cref{exp:bndlayer_notsym}, convection dominated problem. Computational cost vs. augmented norm error and speed-up in function of the error.}
	\label{fig:bndlayer_notsym_effenerr_speedup}
\end{figure}

For completeness, we plot in \cref{fig:bndlayer_notsym_etaeffind} the effectivity index of $\eta$. We see that it is completely off. This illustrates that this estimator does not capture the convective error and is hence not appropriate for convection dominated problems. Then, we run again the same experiment but considering different diffusion coefficients $\varepsilon=10^{-4}, 10^{-6}, 10^{-8}$ and display the effectivity indexes of $\tilde\eta$ in \cref{fig:bndlayer_notsym_diff_eps}.
\begin{figure}
	\begin{center}
		\begin{subfigure}[t]{\subfigsize\textwidth}
			\centering
			\begin{tikzpicture}[scale=\plotimscale]
			\begin{axis}[height=\aspectratio*\plotimsized\textwidth,width=\plotimsized\textwidth,legend style={at={(1,1)},anchor=north east},
			xlabel={Iteration $k$}, ylabel={Effectivity index of $\eta$},ymin=0,ymax=800,,label style={font=\normalsize},tick label style={font=\normalsize},legend image post style={scale=\legendmarkscale},legend style={nodes={scale=\legendfontscale, transform shape},draw=none}]
			\addplot[color=OrangeRed,mark=o,line width=\plotlinewidth pt,mark size=\plotmarksizeu pt] table [x=level,y=eff,col sep=comma,select coords between index={6}{19}] 
			{data/bndlayer/SPA2FFM_sing_4_diff_6_b_1_nref_3_lay_21_a_posteriori_data.csv};\addlegendentry{Local}
			\addplot[color=ForestGreen,mark=star,line width=\plotlinewidth pt,mark size=\plotmarksizeu pt] table[x=level,y=eff,col sep=comma,select coords between index={6}{19}] 
			{data/bndlayer/SPA1_sing_4_diff_6_b_1_nref_3_lay_21_a_posteriori_data.csv};\addlegendentry{Classical}
			\end{axis}
			\end{tikzpicture}
			\caption{Effectivity index of $\eta$.}
			\label{fig:bndlayer_notsym_etaeffind}
		\end{subfigure}\hfill
		\begin{subfigure}[t]{\subfigsize\textwidth}
			\centering
			\begin{tikzpicture}[scale=\plotimscale]
			\begin{axis}[height=\aspectratio*\plotimsized\textwidth,width=\plotimsized\textwidth,legend columns=2, ,legend style={at={(1,1)},anchor=north east,draw=none,fill=none},
			xlabel={Iteration $k$}, ylabel={Effectivity index of $\tilde\eta$},ymin=0,ymax=20,label style={font=\normalsize},tick label style={font=\normalsize},legend image post style={scale=\legendmarkscale},legend style={nodes={scale=\legendfontscale, transform shape},draw=none}]
			\addplot[color=ForestGreen,mark=star,line width=\plotlinewidth pt,mark size=\plotmarksizeu pt] table[x=level,y=efft,col sep=comma,select coords between index={6}{19}] 
			{data/bndlayer/SPA2FFM_sing_4_diff_4_b_1_nref_3_lay_21_a_posteriori_data.csv};\addlegendentry{$\varepsilon=10^{-4}$}
			\addplot[color=OrangeRed,mark=o,line width=\plotlinewidth pt,mark size=\plotmarksizeu pt] table [x=level,y=efft,col sep=comma,select coords between index={6}{19}] 
			{data/bndlayer/SPA2FFM_sing_4_diff_6_b_1_nref_3_lay_21_a_posteriori_data.csv};\addlegendentry{$\varepsilon=10^{-6}$}
			\addlegendimage{empty legend}\addlegendentry{}
			\addplot[color=NavyBlue,mark=triangle,line width=\plotlinewidth pt,mark size=\plotmarksizeu pt] table[x=level,y=efft,col sep=comma,select coords between index={6}{19}]  	
			{data/bndlayer/SPA2FFM_sing_4_diff_8_b_1_nref_3_lay_21_a_posteriori_data.csv};\addlegendentry{$\varepsilon=10^{-8}$}
			\end{axis}
			\end{tikzpicture}
			\caption{Effectivity index of $\tilde\eta$ for different diffusion coefficients $\varepsilon$.}
			\label{fig:bndlayer_notsym_diff_eps}
		\end{subfigure}
	\end{center}
	\caption{\Cref{exp:bndlayer_notsym}, convection dominated problem. Effectivity index of $\eta$ and of $\tilde\eta$ but for different diffusion coefficients $\varepsilon$.}
	\label{fig:bndlayer_notsym_etaeffind_diff_eps}
\end{figure}
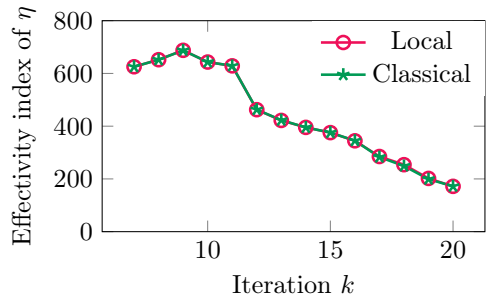
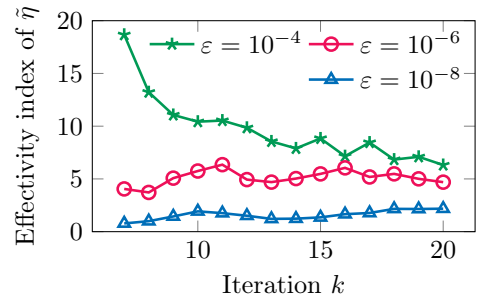

\subsection{A nonlinear nonsmooth problem with multiple local structures}\label{exp:nonlin}
We conclude with an experiment on a nonlinear nonsmooth problem, where the diffusion tensor is solution dependent and has multiple discontinuities, hence the solution presents several local structures. More precisely, we solve \cref{eq:elliptic} with $\Omega=[-3/2,3/2]\times [-3/2,3/2]$, $\bbeta=-(1,1)^\top$, $\mu=1$ and $f(\bx)=\nld{\bx}^2$. The diffusion tensor is $A(u,\bx)=A_1(u)A_2(\bx)$, with $A_1(u)=1/\sqrt{1+u^2}$. We divide $\Omega$ in nine squares of size $1/2\times 1/2$ and $A_2(\bx)$ alternates between $1$ and $0.01$, in a checkerboard-like manner. A reference solution is displayed in \cref{fig:nonlinsol}.

\cref{thm:energynormbound,thm:augmentednormbound} do not apply straightforwardly as the problem is nonlinear. Nevertheless, \cref{alg:local} can be used in combination with a Newton scheme as it is shown in \cite{AbR19}. In this experiment we investigate the efficiency of the error estimators in identifying the local subdomains for a nonlinear nonsmooth problem with multiple local structures. Starting with a $32\times 32$ elements mesh, we run the code and let it automatically select the subdomains for twenty iterations. We do the same with the classical \cref{alg:classical} and compare the results in \cref{fig:nonlin_eff}, where we display the cost of the Newton method versus the error, computed in energy norm, against a reference solution. We remark as the local method is faster.
 
\begin{figure}
	\begin{center}
		\begin{subfigure}[t]{\subfigsize\textwidth}
			\centering
			\includegraphics[trim=0cm 0cm 0cm 0cm, clip, width=0.7\textwidth]{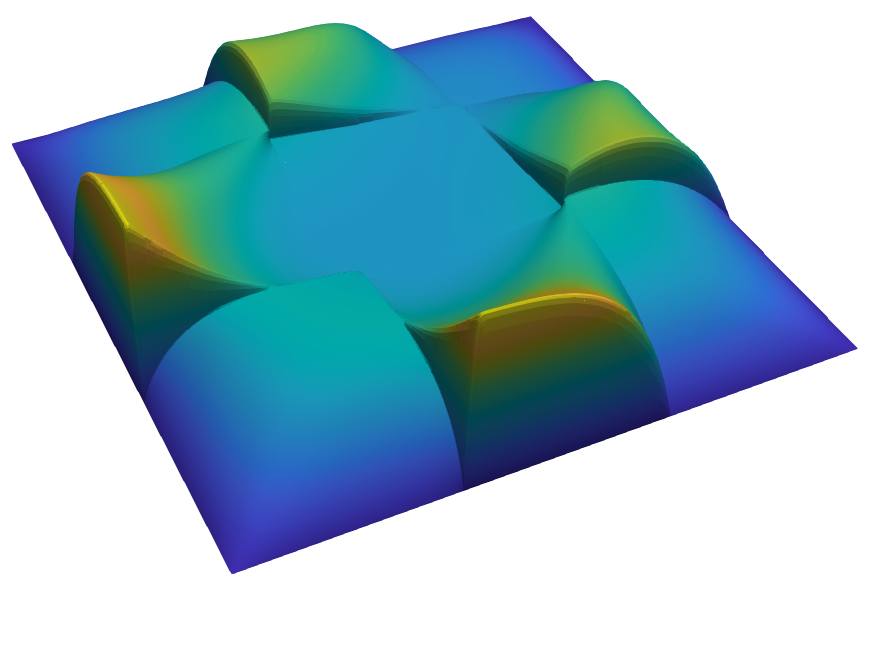}
			\caption{Solution $u(\bx)$ of the nonlinear nonsmooth problem.}
			\label{fig:nonlinsol}
		\end{subfigure}
		\begin{subfigure}[t]{\subfigsize\textwidth}
			\centering
			\begin{tikzpicture}[scale=\plotimscale]
				\begin{loglogaxis}[height=\aspectratio*\plotimsized\textwidth,width=\plotimsized\textwidth, x dir=reverse,legend style={at={(1,0)},anchor=south east,fill=none,draw=none},
					xlabel={Energy norm error.}, ylabel={Newton cost [sec.]},log basis x={2},label style={font=\normalsize},tick label style={font=\normalsize},legend image post style={scale=\legendmarkscale},legend style={nodes={scale=\legendfontscale, transform shape},draw=none}]
					\addplot[color=OrangeRed,mark=o,line width=\plotlinewidth pt,mark size=\plotmarksizeu pt] table [x=erren,y=linsystemtot,col sep=comma] 
					{data/nonlin/spa_2_nref_5_nlev_20_nlay1_2_nlay2_2_a_posteriori_data.csv};\addlegendentry{Local}
					\addplot[color=ForestGreen,mark=star,line width=\plotlinewidth pt,mark size=\plotmarksizeu pt] table[x=erren,y=linsystemtot,col sep=comma] 
					{data/nonlin/spa_1_nref_5_nlev_20_a_posteriori_data.csv};\addlegendentry{Classical}
				\end{loglogaxis}
			\end{tikzpicture}
			\caption{Newton cost versus energy norm error.}
			\label{fig:nonlin_eff}
		\end{subfigure}
	\end{center}
	\caption{Solution $u(\bx)$ and efficiency experiment on the nonlinear nonsmooth problem of \cref{exp:nonlin}.}
\end{figure}
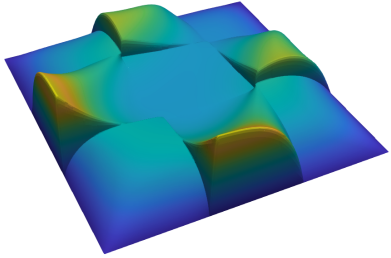
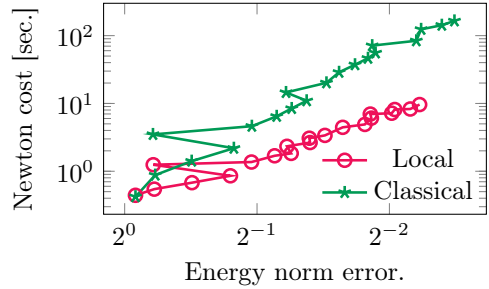

\section{Conclusion}
In this paper we have derived a local adaptive discontinuous Galerkin method for the scheme introduced in \cite{AbR19}. The scheme, defined in \cref{sec:localg}, relies on a coarse solution which is successively improved by solving a sequence of localized elliptic problems in confined subdomains, where the mesh is refined. Starting from error estimators for the symmetric weighted interior penalty Galerkin scheme based on conforming potential and fluxes reconstructions, allowing for flux jumps across the subdomains boundaries we have derived new estimators for the local method and proved their reliability in \cref{thm:energynormbound,thm:augmentednormbound}. An important property of the original estimators (for nonlocal schemes) is conserved: the absence of unknown constants.
Numerical experiments confirm the error estimators' effectivity for singularly perturbed convection-reaction dominated problems and illustrate the efficiency of the local scheme when compared to a classical adaptive algorithm, where at each iteration the solution on the whole computational domain must be recomputed. We also showed that the growth of boundary error indicators (the reason why efficiency cannot be proved in general) can be monitored in order to switch from local to a nonlocal method. Switching automatically from local to classical scheme, based on the indicators $\eta_{\Gamma,1}$, $\eta_{\Gamma,2}$, could be easily integrated in a finite element code. Testing such an integrated code could be of interest to investigate in the future.

\section*{Acknowledgments} The authors are partially supported by the Swiss National Science Foundation, under grant No. $200020\_172710$.


\end{document}